\documentclass{amsart}

\usepackage{amsfonts}
\usepackage{amssymb}
\usepackage{amstext}
\usepackage{graphicx}
\usepackage{tikz}
\usepackage{caption}
\usepackage[]{amscd}
\usepackage{hyperref}
\newtheorem{theo}{Theorem}
\newtheorem{lem}{Lemma}

\newtheorem{remark}{Remark}
\newtheorem{proposition}{Proposition}
\title{Elliptic fibrations on the modular surface associated to $\Gamma_1(8)$}

\begin{document}

\author{Marie Jos\'e Bertin \& Odile Lecacheux}


\email{bertin@math.jussieu.fr, lecacheu@math.jussieu.fr}
\date{\today}

\curraddr{Universit\'e Pierre et Marie Curie (Paris 6), Institut
de Math\'ematiques, 4 Place Jussieu, 75005 PARIS, France}

\begin{abstract} We give all the elliptic fibrations of the $K3$ surface associated to the modular group $\Gamma_1(8)$.
\end{abstract}

\maketitle


\section{Introduction}

Stienstra and Beukers \cite{SB} considered the elliptic pencil
$$xyz+\tau (x+y)(x+z)(y+z)=0$$

and the associated $K3$ surface $\mathcal B$ for $\tau =t^2$, double cover of the modular surface for the modular group $\Gamma_0(6)$. With the help of its $L$-series, they remarked that this surface should carry an elliptic pencil exhibiting it as the elliptic modular surface for $\Gamma_1(8)$ and deplored it was not visible in the previous model of $\mathcal B$.

Later on, studying the link between the logarithmic Mahler measure of some $K3$ surfaces and their $L$-series, Bertin considered in \cite{Ber} $K3$ surfaces of the family

$$(Y_k) \,\,\,\,\,\,\,\,X+\frac {1}{X}+Y+\frac {1}{Y}+Z+\frac {1}{Z}=k.$$
 
For $k=2$, Bertin proved that the corresponding $K3$ surface $Y_2$ is singular (i.e. its Picard rank is $20$) with transcendental lattice 

$$
\left (
\begin{matrix}
2 & 0 \\
0 & 4
\end{matrix}
\right ).
$$

Bertin noticed that $Y_2$ was nothing else than $\mathcal B$, corresponding to the elliptic fibration $X+Y+Z=s$ and $1/\tau =(s-1)^2$. Its singular fibers are of Dynkin type $A_{11}$, $A_5$, $2A_1$ and Kodaira type $I_{12}$, $I_6$, $2I_2$, $2I_1$. Its Mordell-Weil group is the torsion group $\mathbb Z /6 \mathbb Z$.

Using an unpublished result of Lecacheux (see also section 7 of this paper), Bertin showed also that $Y_2$ carries the structure of the modular elliptic surface for $\Gamma_1(8)$. In that case, it corresponds to the elliptic fibration of $Y_2$ with parameter $Z=s$. Its singular fibers are of Dynkin type $2A_7$, $A_3$, $A_1$ and Kodaira type $2I_8$, $I_4$, $I_2$, $2I_1$. Its Mordell-Weil group is the torsion group $\mathbb Z/8 \mathbb Z$.

Interested in $K3$ surfaces with Picard rank $20$ over $\mathbb Q$, Elkies proved in \cite{El1} that their transcendental lattices are primitive of class number one. In particular, he gave in \cite{El2} a list of $11$ negative integers $D$ for which there is a unique $K3$ surface $X$ over $\mathbb Q$ with N\'eron-Severi group of rank $20$ and discriminant $-D$ consisting entirely of classes of divisors defined over $\mathbb Q$.
For $D=-8$, he gave an explicit model of an elliptic fibration with $E_8$ ($=II^*$) fibers at $t=0$ and $t=\infty$ and an $A_1$ ($=I_2$) fiber at $t=-1$

$$y^2=x^3-675x+27(27t-196+\frac {27}{t}).$$

For this fibration, the Mordell-Weil group has rank $1$ and no torsion.

Independently, Sch\"{u}tt proved in \cite{Sch} the existence of $K3$ surfaces of Picard rank $20$ over $\mathbb Q$ and gave for the discriminant $D=-8$ an elliptic fibration with singular fibers $A_3$, $E_7$, $E_8$ ( $I_4$, $III^*$, $II^*$) and Mordell-Weil group equal to $(0)$.
For such a model, you can refer to \cite{Sch1}.

Recall also that Shimada and Zhang gave in \cite{Shim} a list, without equations but with their Mordell-Weil group, of extremal elliptic $K3$ surfaces. In particular, there are $14$ extremal elliptic $K3$ surfaces with transcendental lattice
$$
\left (
\begin{matrix}
2  &   0 \\
0  &   4
\end{matrix}
\right ).
$$

At last, we mention Beukers and Montanus who worked out the semi-stable, extremal, elliptic fibrations of $K3$ surfaces \cite{BM}.

As announced in the abstract, the aim of the paper is to determine all the elliptic fibrations with section on the modular surface associated to $\Gamma_1(8)$ and give for each fibration a Weierstrass model. Thus we recover all the extremal fibrations given by Shimada and Zhang and also fibrations of Bertin, Elkies, Sch\"{u}tt and Stienstra-Beukers mentioned above.

The paper is divided in two parts. In the first sections we use Nishiyama's method, as explained in \cite{Nis} and \cite{SS}, to determine all the elliptic fibrations of $K3$ surfaces with a given transcendental lattice. The method is based on lattice theoretical ideas. We prove the following theorem
\begin{theo}
There are $30$ elliptic fibrations with section, distinct up to isomorphism, on the elliptic surface
$$X+\frac {1}{X}+Y+\frac {1}{Y}+Z+\frac {1}{Z}=2.$$

They are listed in Table 1 with the rank and torsion of their Mordell-Weil group. 

The list consists of $14$ fibrations of rank $0$, $13$ fibrations of rank $1$ and $3$ fibrations of rank $2$.
\end{theo}

In the second part, i.e. sections 7 to 10, we first explain that $Y_2$ is the modular surface associated to the modular group $\Gamma_1(8)$. From one of its fibrations we deduce that it is the unique $K3$ surface $X$ over $\mathbb Q$ with N\'eron-Severi group of rank $20$ and discriminant $-8$, all of its classes of divisors being defined over $\mathbb Q$.

Then, for each fibration, we determine explicitly a Weierstrass model, with generators of the Mordell-Weil group. 

We first use the $8$-torsion sections of the modular fibration to construct the $16$ first fibrations. Their parameters belong to a special group generated by  $10$ functions on the surface. This construction is similar to the one developed for $\Gamma_1(7)$ by Harrache and Lecacheux in \cite{HL}. The next fibrations are obtained by classical methods of gluing and breaking singular fibers. The last ones are constructed by adding a vertex to the graph of the modular fibration.

The construction of some of the fibrations can be done also for the other $K3$ surfaces $Y_k$ of the family. Thus we hope to find for them fibrations of rank $0$ and perhaps obtain more easily the discriminant of the transcendental lattice for singular $K3$ members.

\section{Definitions}

An integral symmetric bilinear form or a lattice of rank $r$ is a free $\mathbb Z$-module $S$ of rank $r$ together with a symmetric bilinear form
$$
\begin{matrix}
b: & S\times S & \rightarrow & \mathbb Z\\
 & (x,x') & \mapsto &  b(x,x').
\end{matrix}
$$

Tensoring by $\mathbb R$ we get the real bilinear form $S_{\mathbb R}$ associated to $S$.

If $S$ is a non-degenerate lattice, we write the signature of $S$, $\hbox{sign}(S)=(t_+,t_-)$. An indefinite lattice of signature $(1,t_-)$ or $(t_+,1)$ is called an hyperbolic lattice.

A homomorphism of lattices $f:S \rightarrow S'$ is a homomorphism of the abelian groups such that $b'(f(x,.f(y))=b(x,y)$ for all $x,y \in S$.

An injective (resp. bijective) homomorphism of lattices is called an {\bf{embedding}} (resp. an isometry). The group of isometries of a lattice $S$ into itself is denoted by $O(S)$ and called the orthogonal group of $S$.

An embedding $i:S \rightarrow S'$ is called {\bf{primitive}} if $S'/i(S)$ is a free group.

A sublattice is a subgroup equipped with the induced bilinear form. A sublattice $S'$ of a lattice $S$ is called primitive if the identity map $S' \rightarrow S$ is a primitive embedding.
The {\bf{primitive closure}} of $S$ inside $S'$ is defined by
$$\overline{S} =\{x\in S' / mx\in S \hbox{  for some positive integer }m \}.$$

A lattice $M$ is an{\bf{ overlattice}} of $S$ if $S$ is a sublattice of $M$ such that the index $[M:S]$ is finite.

Two embeddings $i:S \rightarrow S'$ and $i':S \rightarrow S'$ are called isomorphic if there exists an isometry $\sigma \in O(S')$ such that $i'=\sigma \circ i$. 

By $S_1\oplus S_2$ we denote the orthogonal sum of two lattices defined in the standard way. We write $S^n$ for the orthogonal sum of $n$ copies of a lattice $S$. The {\bf{ orthogonal complement of a sublattice}} $S$ of a lattice $S'$ is defined in the usual way and is denoted by $(S)_{S'}^{\perp}$,
$$(S)_{S'}^{\perp}=\{x\in S'/ b(x,y)=0 \,\,\,\hbox{for all  }y\in S\}.$$

If $\underset{-}{e}=(e_1,\ldots, e_r)$ is a $\mathbb Z$-basis of a lattice $S$, then the matrix 
$G(\underset{-}{e})=(b(e_i,e_j))$ is called the Gram matrix of $S$ with respect to $\underset{-}{e}$.

For every integer $m$ we denote by $S[m]$ the lattice obtained from a lattice $S$ by multiplying the values of its bilinear form by $m$.

A lattice $S$ is called {\bf{even}} if $x^2:=b(x,x)$ is even for all $x$ from $S$. In this case the map $x \rightarrow x^2$ is a quadratic form on $S$ and from it we can recover the symmetric bilinear form on $S$.

For any integer $n$ we denote by $\langle n \rangle$ the lattice $\mathbb Z e$ where $e^2=n$.

\section{Discriminant forms}

Let $L$ be a non-degenerate lattice. 
The {\bf{dual lattice}} $L^*$ of $L$ is defined by
$$L^* =\{x\in L \otimes \mathbb Q /\,\,\, b(x,y)\in \mathbb Z \hbox{  for all }y \in L \}$$
Obviously $L^*$ is an overlattice of $L$ but its form takes non-integer values too. The canonical bilinear form on $L^*$ induced by $b$ is denoted by the same letter.

The {\bf{discriminant group $G_L$}} is defined by
$$G_L:=L^*/L.$$
This group is finite if and only if $L$ is non-degenerate. In the latter case, its order is equal to the absolute value of the lattice determinant $\mid \det (G(\underset{-}{e})) \mid$ for any basis $\underset{-}{e}$ of $L$.

A lattice $L$ is {\bf{unimodular}} if $G_L$ is trivial.

Let $G_L$ be the {\bf{discriminant group }} of a non-degenerate lattice $L$. The bilinear form of $L$ extends naturally to a $\mathbb Q$-valued symmetric bilinear form on $L^*$ and induces a symmetric bilinear form 
$$b_L: G_L \times G_L \rightarrow\mathbb Q / \mathbb Z.$$

If $L$ is even, then $b_L$ is the symmetric bilinear form associated to the quadratic form 
$$q_L: G_L \rightarrow \mathbb Q/2\mathbb Z$$
defined by
$$q_L(x+L)=x^2+2\mathbb Z.$$

The latter means that $q_L(na)=n^2q_L(a)$ for all $n\in \mathbb Z$, $a\in G_L$ and $b_L(a,a')=\frac {1}{2}(q_L(a+a')-q_L(a)-q_L(a'))$, for all $a,a' \in G_L$, where $\frac {1}{2}:\mathbb Q/2 \mathbb Z \rightarrow \mathbb Q / \mathbb Z$ is the natural isomorphism.

The pair{\bf{ $(G_L,b_L)$ (resp. $(G_L,q_L)$) is called the discriminant bilinear (resp. quadratic) form of $L$}}.

\section{Root lattices}

In this section we recall only what is needed for the understanding of the paper. For proofs and details one can refer to Bourbaki \cite{Bo} and Martinet \cite{Ma}.
 
Let $L$ be a negative-definite even lattice. We call $e\in L$ a root if $q_L(e)=-2$. Put $\Delta(L):=\{e\in L/ q_L(e)=-2\}$. Then the sublattice of $L$ spanned by $\Delta(L)$ is called the{\bf{ root type}} of $L$ and is denoted by {\bf{$L_{\hbox{root}}$}}.

If $e\in \Delta(L)$, then an isometry $R_e$ of $L$ is defined by
$$R_e(x)=x+b(x,e)e.$$

We call $R_e$ the reflection associated with $e$. The subgroup of $O(L)$ generated by $R_e$ ($e\in \Delta(L)$) is called the {\bf{Weyl group}} of $L$ and is denoted by $W(L)$.

\bigskip

The{\bf{ lattices $A_m$ ($m\geq 1$), $D_n$ ($n\geq 4$), $E_p$ ($p=6,7,8$) defined by the following Dynkin diagrams are called the root lattices}}: we use Bourbaki's definitions \cite{Bo}.

\begin{table}[h]
\begin{center}
\begin{tabular}{cc}
$A_n=\langle a_1,a_2, \ldots ,a_n \rangle $&
\begin{tabular}{c}  \begin{tikzpicture}[scale=0.7]%
\draw[fill=black](0.1,1.7039)circle (0.05cm)node [below] {$a_1$};
\draw[fill=black](2.9,1.7039)circle (0.05cm)node [below] {$a_2$};
\draw[fill=black](5.7,1.7039)circle (0.05cm)node [below] {$a_3$};
\draw[fill=black](9.9,1.7039)circle (0.05cm)node [below]{$a_n$};
\draw(0.1,1.7039)--(6.81,1.7039);
\draw[dashed](7.3,1.7039)--(8.6,1.7039);
\draw(9.01,1.7039)--(9.9,1.7039);
\end{tikzpicture} \end{tabular} \\
$D_l=\langle d_1,d_2, \ldots ,d_l \rangle $&
\,\begin{tabular}{c}  \begin{tikzpicture}[scale=0.7]%
\draw[fill=black](0.1,1.7039)circle (0.05cm)node [below] {$d_l$};
\draw[fill=black](2.9,1.7039)circle (0.05cm)node [below] {$d_{l-2}$};
\draw[fill=black](5.7,1.7039)circle (0.05cm)node [below] {$d_{l-3}$};
\draw[fill=black](9.9,1.7039)circle (0.05cm)node [below]{$d_1$};
\draw[fill=black](2.9,3.10)circle (0.05cm)node [above]{$d_{l-1}$};
\draw(0.1,1.7039)--(6.81,1.7039);
\draw[dashed](7.3,1.7039)--(8.6,1.7039);
\draw(9.01,1.7039)--(9.9,1.7039);
\draw(2.9,1.7039)--(2.9,3.10);
\end{tikzpicture} \end{tabular} \\
$E_p=\langle e_1,e_2, \ldots ,e_p \rangle $&
\begin{tabular}{c}  \begin{tikzpicture}[scale=0.7]%
\draw[fill=black](0.1,1.7039)circle (0.05cm)node [below] {$e_1$};
\,\,\draw[fill=black](2.9,1.7039)circle (0.05cm)node [below] {$e_{3}$};
\draw[fill=black](5.7,1.7039)circle (0.05cm)node [below] {$e_{4}$};
\draw[fill=black](9.9,1.7039)circle (0.05cm)node [below]{$e_p$};
\draw[fill=black](5.7,3.10)circle (0.05cm)node [above]{$e_{2}$};
\draw(0.1,1.7039)--(6.81,1.7039);
\draw[dashed](7.3,1.7039)--(8.6,1.7039);
\draw(9.01,1.7039)--(9.9,1.7039);
\draw(5.7,1.7039)--(5.7,3.10);
\end{tikzpicture} \end{tabular} \\
\end{tabular}

\end{center}
\end{table}

All the vertices $a_j$, $d_k$, $e_l$ are roots and two vertices $a_j$ and $a_j'$ are joined by a line if and only if $b(a_j,a_j')=1$.

Denote $\epsilon_i$ the vectors of the canonical basis of $\mathbb R^n$ with the usual scalar product.

\subsection{{\bf{$A_l^{*}/A_l$}}}

We can represent $A_l$ by the set of points in $\mathbb R^{l+1}$ with integer coordinates whose sum is zero.

Set $a_i=\epsilon_i-\epsilon_{i+1}$ and define
$${\alpha}_l=\epsilon_1-\frac {1}{l+1} \sum_{j=1}^{l+1} \epsilon_j=\frac {1}{l+1} \sum_{j=1}^{l}(l-j+1)a_j.$$

One can show that

$$A_l^*=\langle A_l,{\alpha_l} \rangle, \,\,\,\, A_l^*/A_l  \simeq \mathbb Z /(l+1)\mathbb Z \,\,\,\,\hbox{and}\,\,\,q_{A_l}(\alpha_l)=\left( -\frac {l}{l+1} \right).$$

\subsection{{\bf{$D_l^*/D_l$}}}
We can represent $D_l$ as the set of points of $\mathbb R^l$ with integer coordinates of even sum and define

$$
\begin{array}{lll}
{\delta}_{l} & = &\frac {1}{2}(\sum_{i=1}^l \epsilon_i)=\frac {1}{2}\left( \sum_{i=1}^{l-2}id_i+\frac {1}{2}(l-2)d_{l-1}+\frac {1}{2}ld_l \right)\\[15pt]
\overline{\delta}_l & = &\epsilon_1 =\sum_{i=1}^{l-2} d_i+\frac {1}{2}(d_{l-1}+d_l)\\[15pt]
\tilde{\delta}_{l} & = & \delta_l-\epsilon_l =\frac {1}{2}\left( \sum_{i=1}^{l-2}id_i+\frac {1}{2}ld_{l-1}+\frac {1}{2}(l-2)d_l  \right)
\end{array}
$$

One can show that
$$D_l^*= \langle \epsilon_1, \ldots , \epsilon_l, \delta_l \rangle.$$

Then, for $l$ odd

$$ D_l^* /D_l \simeq \mathbb Z /4 \mathbb Z =\langle {\delta_l} \rangle, \,\, \,\,\,\, \overline{\delta}_l\equiv 2 {\delta_l}\,\,\,\,\hbox{and}\,\,\,\,\,\tilde{\delta}_{l}\equiv 3 {\delta}_l\,\,\, \hbox{mod.}\,\,\,D_l$$
and for $l$ even
$$D_l^* /D_l \simeq \mathbb Z /2 \mathbb Z \times \mathbb Z /2 \mathbb Z,$$
the $3$ elements of order $2$ being the images of ${\delta}_l$, $\tilde{\delta}_{l}$ and $\overline{\delta}_l$.

Moreover,
$$q_{D_l}(\delta_l)=\left ( - \frac {l}{4} \right ), \,\,\,\,\,\,q_{D_l}(\overline{\delta}_l)=(-1), \,\,\,\,\,\,b_{D_l}(\delta_l,\overline{\delta}_l)=-\frac {1}{2}.$$

\subsection{{\bf{$E_6^*/E_6$}}}

We can represent $E_6$ as a lattice in $\mathbb R^8$ generated by the $6$ vectors $e_i$,

$$
e_1  =  \frac {1}{2}(\epsilon_1 +\epsilon _8)-\frac {1}{2} (\sum_{i=2}^{7}\epsilon_i),  \,\,\,\,e_2= \epsilon_1+\epsilon_2, \,\,\,\, e_i=\epsilon_{i-1}-\epsilon_{i-2}, \,\,\,\, 3\leq i \leq 6.$$

Denote

$$\eta_6 =-\frac {1}{3}(2e_1+3e_2+4e_3+6e_4+5e_5+4e_6).$$
Then

$$E_6^*=\langle E_6, \eta_6 \rangle , \,\,\,\,E_6^*/E_6\simeq \mathbb Z /3 \mathbb Z ,\,\,\,\, q_{E_6}(\eta_6)=\left( -\frac {4}{3} \right ).$$

\subsection{{\bf{$E_7^*/E_7$}}}

We can represent $E_7$ as a lattice in $\mathbb R^8$ generated by the $6$ previous vectors $e_i$ and $e_7=\epsilon_6 -\epsilon_5$.

Denote
$$\eta_7 =-\frac {1}{2}(2e_1+3e_2+4e_3+6e_4+5e_5+4e_6+3e_7),$$
then

$$E_7^*=\langle E_7, \eta_7 \rangle, \,\,\,\,E_7^*/E_7 \simeq \mathbb Z /2 \mathbb Z , \,\,\,\,q_{E_7}(\eta_7)=\left(-\frac {3}{2} \right).$$

\subsection{{\bf{$E_8^*/E_8$}}}

We can represent $E_8$ as the subset of points with coordinates $\xi_i$ satisfying
$$2\xi_i \in \mathbb Z, \,\,\,\,\xi_i-\xi_j \in \mathbb Z, \,\,\,\, \sum_{i=1}^{\infty} \xi_i \in 2 \mathbb Z.$$
Then $E_8^*=E_8$ and $E_8^*/E_8=(0)$.

\section{Elliptic fibrations}

Before giving a complete classification of the elliptic fibrations on the $K3$ surface $Y_2$, we recall briefly some useful facts concerning $K3$ surfaces. For more details see \cite{Ba} \cite{Y}.

\subsection{$K3$ surfaces and elliptic fibrations}

A $K3$ surface $X$ is a smooth projective complex surface with
\[
K_X=\mathcal O_X \,\,\,\,\, \text{and} \,\,\,\,\, H^1(X,\mathcal O_X)=0.
\]
If $X$ is a{\bf{ $K3$ surface}}, then $H^2(X,\mathbb{Z})$ is torsion free. With the cup product,  $H^2(X,\mathbb{Z})$
has the structure of an even lattice. By the Hodge index theorem it has signature $(3,19)$ and by Poincar\'e duality it is unimodular and  $H^2(X,\mathbb{Z})=U^3\oplus E_8(-1)^2.$

The {\bf{N\'eron-Severi group $NS(X)$}} (i.e. the group of line bundles modulo algebraic equivalence), with   the intersection pairing, is a lattice of signature $(1,\rho (X)-1)$, where $\rho(X)$ is the Picard number of $X$. The natural embedding $NS(X) \hookrightarrow H^2(X,\mathbb{Z})$ is a primitive embedding of lattices.

\bigskip

If $C$ is an smooth projective curve over an algebraically closed field $K$, an {\bf{elliptic surface}} $\Sigma$  
over $C$ is an smooth surface with a surjective morphism 
\[
f: \Sigma \rightarrow C
\]
such that almost all fibers are smooth curves of genus 1 and no fiber contains exceptional curves of the first kind.

The morphism $f$ defines an elliptic fibration on $\Sigma$.

We suppose also that every elliptic fibration has a section and so a Weierstrass form. Thus we can consider the generic fiber as an elliptic curve $E$ on $K(C)$ choosing a section as the zero section $\bar{O}$.
In the case of $K3$ surfaces, $C=\mathbb{P}^1$.

The singular fibers were classified by N\'eron \cite{Ne} and Kodaira \cite{Ko}. They are union of irreductible components with multiplicities; each component is a smooth rational curve with self-intersection $-2$.
The singular fibers are classified in the following Kodaira types:
\begin{itemize}
\item
two infinite series $I_n (n>1)$ and $I_n^* (n \geq 0)$ 
\item
five types $III,IV,II^*,I
II^*,IV^*$.
\end{itemize}
The dual graph of these components (a vertex for each component, an edge for each intersection point of two components)  is an extended Dynkin diagram of type $\tilde{A}_n$, $\tilde{D}_l$, $\tilde{E}_p$. Deleting the zero component (i.e. the component meeting the zero section) gives the Dynkin diagram graph $A_n$, $D_l$, $E_p$.   
 We draw the most useful diagrams, with the multiplicity of the components, the zero component being represented by a circle.

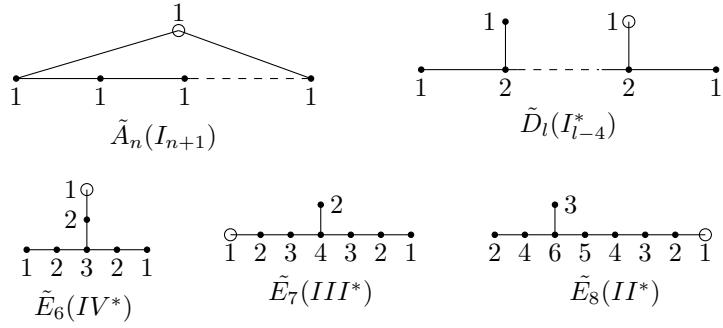
\begin{table}[h]
\begin{center}
\begin{tabular}{ccc}
 
 \begin{tabular}{c}  \begin{tikzpicture}[scale=0.4]%
\draw[fill=black](0.1,1.7039)circle (0.09cm)node [below] {$1$};
\draw[fill=black](2.9,1.7039)circle (0.09cm)node [below] {$1$};
\draw[fill=black](5.7,1.7039)circle (0.09cm)node [below] {$1$};
\draw[fill=black](9.9,1.7039)circle (0.09cm)node [below]{$1$};
\draw(0.1,1.7039)--(5.9,1.7039);
\draw[dashed](5.9,1.7039)--(9.7,1.7039);
\draw(9.7,1.7039)--(9.9,1.7039);
\draw(0.1,1.7039)--(5.5,3.3);
\draw(5.5,3.3)--(9.9,1.7039);
\draw(5.5,3.3) circle (0.2cm)node[above]{$1$};
\draw (5,-1)node [above]{$\tilde{A}_n (I_{n+1})$};
\end{tikzpicture} \end{tabular}
 & & 
 \begin{tabular}{c}  \begin{tikzpicture}[scale=0.4]%
\draw[fill=black](0.1,1.7039)circle (0.09cm)node [below] {$1$};
\draw[fill=black](2.9,1.7039)circle (0.09cm)node [below] {$2$};
\draw[fill=black](7,1.7039)circle (0.09cm)node [below] {$2$};
\draw[fill=black](9.9,1.7039)circle (0.09cm)node [below]{$1$};
\draw(0.1,1.7039)--(3.3,1.7039);
\draw[dashed](3.3,1.7039)--(6.01,1.7039);
\draw(6.1,1.7039)--(9.9,1.7039);
\draw(2.9,1.7039)--(2.9,3.3);
\draw[fill=black](2.9,3.3) circle (0.09cm)node[left] {$1$};
\draw(7,3.3)--(7,1.7039);
\draw(7,3.3) circle (0.2cm)node[left]{$1$};
\draw(5,-1) node[above]{$\tilde{D}_l (I_{l-4}^*)$};
\end{tikzpicture} \end{tabular}
  \end{tabular}

\begin{tabular}{lll}
 
 \begin{tabular}{c}  \begin{tikzpicture}[scale=0.4]%
\draw[fill=black](0.1,1.7039)circle (0.09cm)node [below] {$1$};
\draw[fill=black](1.1,1.7039)circle (0.09cm)node [below] {$2$};
\draw[fill=black](2.1,1.7039)circle (0.09cm)node [below] {$3$};
\draw[fill=black](3.1,1.7039)circle (0.09cm)node [below] {$2$};
\draw[fill=black](4.1,1.7039)circle (0.09cm)node [below] {$1$};
\draw[fill=black](2.1,2.7039)circle (0.09cm)node [left] {$2$};
\draw(2.1,3.7039)circle (0.20cm)node [left] {$1$};
\draw(0.1,1.7039)--(4.1,1.7039);
\draw(2.1,1.7039)--(2.1,3.7039);
\draw (2,-1) node[above]{$\tilde{E}_6 (IV^*)$};
\end{tikzpicture} \end{tabular}&

\begin{tabular}{c}  \begin{tikzpicture}[scale=0.4]%
\draw(0.1,1.7039)circle (0.2cm)node [below] {$1$};
\draw[fill=black](1.1,1.7039)circle (0.09cm)node [below] {$2$};
\draw[fill=black](2.1,1.7039)circle (0.09cm)node [below] {$3$};
\draw[fill=black](3.1,1.7039)circle (0.09cm)node [below] {$4$};
\draw[fill=black](4.1,1.7039)circle (0.09cm)node [below] {$3$};
\draw[fill=black](5.1,1.7039)circle (0.09cm)node [below] {$2$};
\draw[fill=black](6.1,1.7039)circle (0.09cm)node [below] {$1$};
\draw[fill=black](3.1,2.7039)circle (0.09cm)node [right] {$2$};
\draw(0.1,1.7039)--(6.1,1.7039);
\draw(3.1,1.7039)--(3.1,2.7039);
\draw(3.1,-1)node[above]{$\tilde{E_7} (III^*)$};
\end{tikzpicture} \end{tabular}&
 
\begin{tabular}{c}  \begin{tikzpicture}[scale=0.4]%
\draw[fill=black](0.1,1.7039)circle (0.09cm)node [below] {$2$};
\draw[fill=black](1.1,1.7039)circle (0.09cm)node [below] {$4$};
\draw[fill=black](2.1,1.7039)circle (0.09cm)node [below] {$6$};
\draw[fill=black](3.1,1.7039)circle (0.09cm)node [below] {$5$};
\draw[fill=black](4.1,1.7039)circle (0.09cm)node [below] {$4$};
\draw[fill=black](5.1,1.7039)circle (0.09cm)node [below] {$3$};
\draw[fill=black](6.1,1.7039)circle (0.09cm)node [below] {$2$};
\draw(7.1,1.7039)circle (0.2cm)node [below] {$1$};
\draw[fill=black](2.1,2.7039)circle (0.09cm)node [right] {$3$};
\draw(0.1,1.7039)--(7.1,1.7039);
\draw(2.1,1.7039)--(2.1,2.7039);
\draw(4.1,-1) node [above]{$\tilde{E_8} (II^*)$};
\end{tikzpicture} \end{tabular}

 \end{tabular}
 \end{center}
 \end{table}

The {\bf{trivial lattice $T(X)$}} is the subgroup of the N\'eron-Severi group generated by the zero section and the fibers components.  
More precisely, the trivial lattice is the orthogonal sum 
\[
T(X)=<\bar{O},F>\oplus_{v \in S} T_v
\]

where $\bar{O}$ denotes the zero section, $F$ the general fiber, $S$ the points of $C$ corresponding to the reducible singular fibers and $T_v$ the lattice generated by the fiber components except the zero component. 
 
From this formula we can compute the determinant of $T(X)$.

From Shioda's results on height pairing \cite{Shio} we can define a positive-definite lattice structure on the Mordell-Weil lattice $MWL(X):=E(K(C))/E(K(C))_{tor}$ and get the following proposition.

\begin{proposition}
Let $X$ be a $K3$ surface or more generally any elliptic surface with section. 
We have the relation
\[
|disc(NS(X))| = disc(T(X)) disc( MWL(X))/|E(K)_{tor}|^2.
\]
\end{proposition}  

\bigskip

Moreover since $X$ is a $K3$ surface, the zero section has self-intersection $\bar{O}^2=-\chi (X)=-2$. Hence the zero section $\bar{O}$  and the general fiber $F$ generate an even unimodular lattice, called the hyperbolic plane $U$. 

The trivial lattice $T(X)$ of an elliptic surface is not always primitive in $NS(X)$. Its primitive closure $\overline{T(X)}$ is obtained by adding the torsion sections.

The N\'eron-Severi lattice $NS(X)$ always contains an even sublattice of corank two, the {\bf{ frame $W(X)$}}
$$W(X)=\langle \bar{O},F \rangle ^{\perp} \subset NS(X).$$

\begin{lem}
For any elliptic surface $X$ with section, the frame $W(X)$ is a negative-definite even lattice of rank $\rho (X) -2$.
\end{lem}

Hence, the N\'eron-Severi lattice of a $K3$ surface is an even lattice.

One can read off, the Mordell-Weil lattice, the torsion in the Mordell-Weil group MW and the type of singular fibers from $W(X)$ by

$$MWL(X)=W(X)/{\overline{W(X})_{\hbox{root}}}\,\,\,\,\,\,(MW)_{\hbox{tors}}=\overline{W(X)}_{\hbox{root}}/W(X)_{\hbox{root}}$$ 
$$T(X)=U\oplus W(X)_{\hbox{root}}.$$

We can also calculate the heights of points from the Weierstrass equation \cite{Ku} and test if points generate the Mordell-Weil group, since $disc (NS(X))$ is independant of the fibration.

\subsection{Nikulin and Niemeier's results}

\begin{lem}(Nikulin \cite{Nik}, Proposition 1.4.1)
Let $L$ be an even lattice. Then, for an even overlattice $M$ of $L$, we have a subgroup $M/L$ of $G_L=L^*/L$ such that $q_L$ is trivial on $M/L$. This determines a bijective correspondence between even overlattices of $L$ and subgroups $G$ of $G_L$ such that $q_L\mid_G=0$.
\end{lem}

\begin{lem}(Nikulin \cite{Nik}, Proposition 1.6.1)
Let $L$ be an even unimodular lattice and $T$ a primitive sublattice. Then we have
$$G_T\simeq G_{T^{\perp}} \simeq L/(T\oplus T^{\perp}),\,\,\,\,\,\,\,q_{T^{\perp}}=-q_T.$$
In particular, $ \det T =\det  T^{\perp} =[L:T\oplus T^{\perp}]$.
\end{lem}

\begin{theo} (Nikulin \cite{Nik} Corollary 1.6.2)
Let $L$ and $M$ be even non-degenerate integral lattices such that
$$G_L \simeq G_M,\,\,\,\,\,\,\,q_L=-q_M.$$
Then there exists an unimodular overlattice $N$ of $L\oplus M$ such that

1) the embeddings of $L$ and $M$ in $N$ are primitive

2) $L_N^{\perp}=M$ and $M_N^{\perp}=L$.
\end{theo}

\begin{theo}(Nikulin \cite{Nik} Theorem 1.12.4 )
Let there be given two pairs of nonnegative integers, $(t_{(+)},t_{(-)})$ and  $(l_{(+)},l_{(-)})$. The following properties are equivalent:

a) every even lattice of signature  $(t_{(+)},t_{(-)})$ admits a primitive embedding into some even unimodular lattice of signature  $(l_{(+)},l_{(-)})$;

b)  $l_{(+)}-l_{(-)} \equiv 0 \,\, (\hbox{mod}\,\, 8)$, $t_{(+)} \leq l_{(+)}$, $t_{(-)} \leq l_{(-)}$ and  $t_{(+)}+t_{(-)} \leq \frac {1}{2}(l_{(+)}+l_{(-)})$.

\end{theo}

\begin{theo}(Niemeier \cite{Nie})
A negative-definite even unimodular lattice $L$ of rank $24$ is determined by its root lattice $L_{\hbox{root}}$ up to isometries. There are $24$ possibilities for $L$ and $L/L_{\hbox{root}}$ listed below.
\end{theo}

\begin{center}
\begin{tabular}{|l|r|l|r|}
\hline

$L_{\hbox{root}}$ & $L/L_{\hbox{root}}$ & $L_{\hbox{root}}$ & $L/L_{\hbox{root}}$ \\ \hline
  $E_8^3$    &  $(0)$ & $D_5^{\oplus 2}\oplus A_7^{\oplus 2} $& $\mathbb Z/4\mathbb Z \oplus \mathbb Z /8 \mathbb Z$\\ \hline
$E_8\oplus D_{16}$ & $\mathbb Z /2 \mathbb Z $ & $A_8^{\oplus 3}$ & $\mathbb Z/3\mathbb Z \oplus \mathbb Z / 9 \mathbb Z$ \\ \hline
$E_7^{\oplus 2}\oplus D_{10}$ & $(\mathbb Z/2 \mathbb Z)^2$ & $A_{24}$ & $\mathbb Z/5 \mathbb Z $\\\hline 
$E_7 \oplus A_{17}$ & $\mathbb Z /6 \mathbb Z$ & $A_{12}^{\oplus 2}$ & $\mathbb Z /13 \mathbb Z$\\ \hline
$D_{24}$ & $\mathbb Z /2 \mathbb Z $ & $D_4^{\oplus 6}$ & $(\mathbb Z /2 \mathbb Z)^6$ \\ \hline
$D_{12}^{\oplus 2}$ & $(\mathbb Z /2 \mathbb Z)^2$ & $D_4 \oplus A_5^{\oplus 4}$ & $\mathbb Z /2 \mathbb Z \oplus ( \mathbb Z /6 \mathbb Z )^2$ \\ \hline
$D_8^{\oplus 3 }$ & $(\mathbb Z /2 \mathbb Z)^3$ & $A_6^{\oplus 4}$ & $(\mathbb Z /7 \mathbb Z )^2$ \\ \hline
$D_9 \oplus A_{15} $ & $\mathbb Z /8 \mathbb Z $ & $A_4^{\oplus 6}$ & $(\mathbb Z / 5 \mathbb Z)^3$\\ \hline
$E_6^{\oplus 4}$ &  $(\mathbb Z /3 \mathbb Z)^2$ & $A_3^{\oplus 8}$ & $(\mathbb Z /4 \mathbb Z)^4$ \\ \hline

$E_6 \oplus D_7 \oplus A_{11}$ & $\mathbb Z /12 \mathbb Z $ & $A_2^{\oplus 12}$ & $(\mathbb Z /3 \mathbb Z)^6 $ \\ \hline
$D_6^{\oplus 4}$ & $(\mathbb Z /2 \mathbb Z)^4$ & $A_1^{\oplus 24}$ & $(\mathbb Z /2 \mathbb Z)^{12}$\\ \hline
$D_6 \oplus A_9^{\oplus 2}$ & $\mathbb Z /2 \mathbb Z \oplus \mathbb Z /10 \mathbb Z $ & $0$ & $\Lambda _{24}$\\

\hline
\end{tabular}
\end{center}

The lattices $L$ defined in theorem 4 are called {\bf{Niemeier lattices}}.

\subsection{Nishiyama's method}

Recall that a $K3$ surface may admit more than one elliptic fibration, but up to isomorphism, there is only a finite number of elliptic fibrations \cite{St}.

To establish a complete classification of the elliptic fibrations on the $K3$ surface $Y_2$, we use Nishiyama's method based on lattice theoretic ideas \cite{Nis}. The technique builds on a converse of Nikulin's results.

Given an elliptic $K3$ surface $X$, Nishiyama aims at embedding the frames of all elliptic fibrations into a negative-definite lattice, more precisely into a Niemeier lattice of rank $24$. For this purpose, he first determines an even negative-definite lattice $M$ such that
$$q_M=-q_{NS(X)}, \,\,\,\,\,\,\hbox{rank}(M)+\rho(X)=26.$$
By theorem 2, $M\oplus W(X)$ has a Niemeier lattice as an overlattice for each frame $W(X)$ of an elliptic fibration on $X$. Thus one is bound to determine the (inequivalent) primitive embeddings of $M$ into Niemeier lattices $L$. To achieve this, it is essential to consider the root lattices involved. In each case, the orthogonal complement of $M$ into $L$ gives the corresponding frame $W(X)$.

\subsubsection{The transcendental lattice and argument from Nishiyama paper}
Denote by $\mathbb T(X)$ the transcendental lattice of $X$, i.e. the orthogonal complement of $NS(X)$ in $H^2(X,\mathbb Z)$ with respect to the cup-product,
$$\mathbb T(X)=\hbox{NS}(X)^{\perp} \subset H^2(X,\mathbb Z).$$
In general, $\mathbb T(X)$ is an even lattice of rank $r=22-\rho(X)$ and signature $(2,20-\rho(X))$. Let $t=r-2$. By Nikulin's theorem 3, $\mathbb T(X)[-1]$ admits a primitive embedding into the following indefinite unimodular lattice:
$$\mathbb T(X)[-1] \hookrightarrow U^t \oplus E_8.$$
Then define $M$ as the orthogonal complement of $\mathbb T(X)[-1]$ in $ U^t \oplus E_8$. By construction, $M$ is a negative-definite lattice of rank $2t+8-r=r+4=26-\rho(X)$. 

By lemma 3 the discriminant form satisfies
$$q_M=-q_{\mathbb T(X)[-1]}=q_{\mathbb T(X)}=-q_{\hbox{NS}(X)}.$$
Hence $M$ takes exactly the shape required for Nishiyama's technique.

\subsubsection{Torsion group}
First we classify all the primitive embeddings of $M$ into $L_{\hbox{root}}$. Denote $N:=M^{\perp}$ in $L_{\hbox{root}}$ and $W=M^{\perp}$ into $L$.

If $M$ satisfies $M_{\hbox{root}}=M$, we can apply Nishiyama's results \cite{Nis}.

In particular,
\begin{itemize}
\item $M$ primitively embedded in $L_{\hbox{root}}$ $\Longleftrightarrow $ $M$ primitively embedded in $L$
\item $N/N_{\hbox{root}}$ is torsion-free.
\end{itemize}

Notice that the rank $r$ of the Mordell-Weil group is equal to $\hbox{rk} (W)- \hbox{rk} (W_{\hbox{root}})$ and its torsion part is $\overline{W}_{\hbox{root}}/W_{\hbox{root}}$.

We need also the following lemma.
\begin{lem}
\begin{enumerate}
\item If $\det N=\det M$, then the Mordell-Weil group is torsion-free.
\item If $r=0$, then the Mordell-Weil group is isomorphic to $W/N$.
\item In general, there are the following inclusions
$$ \overline{W}_{\hbox{root}}/W_{\hbox{root}} \subset W/N \subset L/L_{\hbox{root}}.$$
\end{enumerate}
\end{lem}
\section{Elliptic fibrations of $Y_2$}
\begin{theo}
There are $30$ elliptic fibrations with section, unique up to isomorphism, on the elliptic surface $Y_2$. They are listed with the rank and torsion of their Mordell-Weil groups on the table 1.
\end{theo}

\begin{table}
\begin{center}
\begin{tabular}{|l|r|r|r|c|c|c|}
\hline
$L_{\hbox{root}}$ & $L/L_{\hbox{root}}$ &  & & \hbox{Reducible} & \hbox{Rk} & \hbox{Tors.} \\
    &  &    &  & \hbox{fibers} &   &  \\  \hline          \hline
  $E_8^3$    &  $(0)$ & & & & &\\ \hline
&   & $A_1\subset E_8$ & $D_5 \subset E_8$  & $E_7 A_3 E_8$ & $0$ & $(0)$\\ \hline
&   & $A_1\oplus D_5 \subset E_8$ &   & $A_1 E_8  E_8$ & $1$ & $(0)$\\ \hline
  $E_8 D_{16}$    &  $\mathbb Z /{2 \mathbb Z}$ & & & & &\\ \hline
&   & $A_1\subset E_8$ & $D_5 \subset D_{16}$  & $E_7 D_{11}$ & $0$ & $(0)$\\ \hline
&   & $A_1\oplus D_5 \subset E_8$ &   & $A_1 D_{16}$ & $1$ & $\mathbb Z /{2 \mathbb Z}$\\ \hline
&   & $D_5\subset E_8$ & $A_1 \subset D_{16}$  & $A_3 A_1 D_{14}$ & $0$ & $\mathbb Z /{2 \mathbb Z}$\\ \hline
&   & $A_1\oplus D_5 \subset D_{16}$ &   & $E_8 A_1 D_{9}$ & $0$ & $(0)$\\ \hline
  $E_7^2 D_{10}$    &  $(\mathbb Z /{2 \mathbb Z})^2$ & & & & &\\ \hline
&   & $A_1\subset E_7$ & $D_5 \subset D_{10}$  & $E_7 D_6 D_5$ & $0$ & $\mathbb Z /{2 \mathbb Z}$\\ \hline
&   & $A_1\subset E_7$ & $D_5 \subset E_7$  & $ D_6 A_1 D_{10}$ &$1$ & $(0)$\\ \hline
&   & $A_1\oplus D_5 \subset E_7$ &   & $E_7 D_{10}$ & $1$ & $\mathbb Z /{2 \mathbb Z}$\\ \hline
&   & $A_1\oplus D_5 \subset D_{10}$ &   & $E_7 E_7 A_1 A_3$ & $0$ & $\mathbb Z /{2 \mathbb Z}$\\ \hline
&   & $D_5\subset E_7$ & $A_1 \subset D_{10}$  & $A_1 A_1 D_8 E_7$ & $1$ & $\mathbb Z /{2 \mathbb Z}$\\ \hline
 $E_7 A_{17}$    &  $\mathbb Z /{6 \mathbb Z}$ & & & & &\\ \hline
&   & $A_1\oplus D_5 \subset E_7$ &   & $A_{17}$ & $1$ & $\mathbb Z /{3 \mathbb Z}$\\ \hline
&   & $D_5\subset E_7$ & $A_1 \subset A_{17}$  & $A_1 A_{15}$ & $2$ & $(0)$\\ \hline
 $D_{24}$    &  $\mathbb Z /{2 \mathbb Z}$ & & & & &\\ \hline
&   & $A_1\oplus D_5 \subset D_{24}$ &   & $A_1 D_{17}$ & $0$ & $(0)$\\ \hline

 $D_{12}^2$    &  $(\mathbb Z /{2 \mathbb Z})^2$ & & & & &\\ \hline
&   & $A_1 \subset D_{12}$ & $D_5 \subset D_{12}$  & $A_1 D_{10} D_{7}$ & $0$ & $\mathbb Z /{2 \mathbb Z}$\\ \hline
&   & $A_1\oplus D_5 \subset D_{12}$ &   & $A_1 D_5 D_{12}$ & $0$ & $\mathbb Z /{2 \mathbb Z}$\\ \hline
 $D_8^3$    &  $(\mathbb Z /{2 \mathbb Z})^3$ & & & & &\\ \hline
&   & $A_1 \subset D_{8}$ & $D_5 \subset D_{8}$  & $A_1 D_{6}A_3 D_{8}$ & $0$ & $(\mathbb Z /{2 \mathbb Z})^2 $\\ \hline
&   & $A_1\oplus D_5 \subset D_{8}$ &   & $A_1 D_8 D_{8}$ & $1$ & $\mathbb Z /{2 \mathbb Z}$\\ \hline
 $D_9 A_{15}$    &  $\mathbb Z /{8 \mathbb Z}$ & & & & &\\ \hline
&   & $A_1\oplus D_5 \subset D_{9}$ &   & $ A_1 A_1 A_1 A_{15}$ & $0$ & $\mathbb Z /{4 \mathbb Z}$\\ \hline
&   & $D_5 \subset D_{9}$ & $A_1 \subset A_{15}$  & $D_4 A_{13}$ & $1$ & $(0)$\\ \hline
 $E_6^4$    &  $(\mathbb Z /{3 \mathbb Z)^2}$ & & & & &\\ \hline
&   & $A_1 \subset E_6$ & $D_5 \subset E_6$  & $A_5 E_6 E_6$ & $1$ & $\mathbb Z /{3 \mathbb Z}$\\ \hline
 $ A_{11} E_6 D_7$    &  $\mathbb Z /{12 \mathbb Z}$ & & & & &\\ \hline
&   & $A_1 \subset E_6$ & $D_5 \subset D_7$  & $A_5 A_1 A_1 A_{11}$ & $0$ & $\mathbb Z /{6 \mathbb Z}$\\ \hline
&   & $A_1 \subset A_{11}$ & $D_5 \subset D_7$  & $A_9 A_1 A_1 E_6$ & $1$ & $(0)$\\ \hline
&   & $A_1\oplus D_5 \subset D_{7}$ &   & $A_{11}E_6  A_1 $ & $0$ & $\mathbb Z /{3 \mathbb Z}$\\ \hline
&   & $A_1 \subset A_{11}$ & $D_5 \subset E_6$  & $ A_9 D_7$ & $2$ & $(0)$\\ \hline
&   & $D_5 \subset E_6$ & $A_1 \subset D_7$  & $ A_{11} A_1 D_5$ & $1$ & $\mathbb Z /{4 \mathbb Z}$\\ \hline
 $D_6^4$    &  $(\mathbb Z /{2 \mathbb Z})^4$ & & & & &\\ \hline
&   & $A_1 \subset D_6$ & $D_5 \subset D_6$  & $ A_1 D_4 D_6 D_6$ & $1$ & $(\mathbb Z /{2 \mathbb Z})^2$\\ \hline
 $D_6 A_9^2$    &  $\mathbb Z /{2\mathbb Z}\times \mathbb Z /{10\mathbb Z} $ & & & & &\\ \hline
&   & $D_5 \subset D_6$ & $A_1 \subset A_9$  & $ A_{7} A_9$ & $2$ & $(0)$\\ \hline

 $D_5^2 A_7^2$    &  $\mathbb Z /{4 \mathbb Z}\times \mathbb Z /{8\mathbb Z} $ & & & & &\\ \hline
&   & $D_5 \subset D_5$ & $A_1 \subset D_5$  & $A_1  A_{3} A_7 A_7$ & $0$ & $\mathbb Z /{8 \mathbb Z}$\\ \hline
&   & $D_5 \subset D_5$ & $A_1 \subset A_7$  & $D_5 A_5 A_7 $ & $1$ & $(0)$\\ \hline

\end{tabular}
\end{center}
\caption{The elliptic fibrations of $Y_2$}
\end{table}

\begin{proof}
We follow Nishiyama's method. Since 
$$\mathbb T(Y_2)=\left ( \begin{matrix}
                 2 &  0 \\
                 0  & 4
                    \end{matrix}
                   \right ),
                   $$
we get, by Nishiyama's computation \cite{Nis}, $M=D_5 \oplus A_1$. Thus we have to determine all the primitive embeddings of $M$ into the root lattices and their orthogonal complements.

\subsection{The primitive embeddings of $D_5 \oplus A_1$ into root lattices}

\begin{proposition}
There are primitive embeddings of $D_5\oplus A_1$ only into the following $L_{\hbox{root}}$:

$$E_8^{\oplus 3}, \,\, E_8\oplus D_{16}, \,\, E_7^{\oplus 2}\oplus D_{10}, \,\, E_7\oplus A_{17}, \,\, D_8^{\oplus 3}, \,\, D_9\oplus A_{15},$$
$$ E_6^{\oplus 4}, \,\, A_{11}\oplus E_6\oplus D_7, \,\, D_6^{\oplus 4}, \,\, D_6\oplus A_9^{\oplus 2}, \,\,D_5^{\oplus 2}\oplus A_7^{\oplus 2}.$$

\end{proposition}

\begin{proof}
The assertion comes from Nishiyama's results \cite{Nis}. 

The root lattice $A_1$ can be primitively embedded in all $A_n$, $D_l$ and $E_p$, $n\geq 1$, $l\geq 2$, $p=6,7,8$.

The root lattice $D_5$ can be primitively embedded only in $D_l$, $l\geq 5$ and $E_p$, $p=6,7,8$.

The root lattice $D_5 \oplus A_1$ can be primitively embedded only in $D_l$, $l\geq 7$, $E_7$ and $E_8$.

The proposition follows from theorem 4 and the previous facts.
\end{proof}

\begin{proposition}
Up to the action of the Weyl group, the primitive embeddings are given in the following list
\begin{itemize}
\item $A_1=\langle a_n \rangle \subset A_n$
\item $A_1= \langle d_l \rangle  \subset D_l, \,\, l\geq 4$
\item $A_1=\langle e_1 \rangle \subset E_p, \, \, p=6,7,8$
\item $D_5\oplus A_1= \langle d_{l-1}, d_l, d_{l-2}, d_{l-3}, d_{l-4} \rangle \oplus \langle d_{l-6} \rangle \subset D_l, \,\, l\geq 7$
\item $D_5 \oplus A_1 =\langle e_2, e_5, e_4, e_3, e_1 \rangle \oplus \langle e_7 \rangle \subset E_n , \,\, n\geq 7.$
\end{itemize}
\end{proposition}

\begin{proof}
These assertions come from Nishiyama's computations \cite{Nis}. Just be careful of the difference of notations between Nishiyama and us.
\end{proof}

\begin{proposition}
We get the following results about the orthogonal complements of the previous embeddings

\begin{enumerate}

\item
$$(A_1)_{A_n}^{\perp}=L_{n-2}^2=\left ( \begin{matrix}
                                   -2\times 3 & \vline & 2 & 0 & \hdots & 0 \\
                                    \hline \\
                                   2 &\vline & & & & \\
                                    0 & \vline & & & & \\
                                    \vdots & \vline & & A_{n-2} & & \\
                                    0 & \vline & & & & 
                                   \end{matrix} 
                                     \right )$$

with $\det L_{n-2}^2=2(n+1)$ 
\item $$(A_1)_{D_4}^{\perp}=A_1^{\oplus 3}$$
 $$(A_1)_{D_n}^{\perp}=A_1 \oplus D_{n-2}, \,\,\, n\geq 5$$
\item
$$(A_1)_{A_7}^{\perp}=(\langle a_7\rangle)_{A_7}^{\perp} =\langle a_7+2a_6,a_5,a_4,a_3,a_2,a_1\rangle$$
$$\alpha_7 \in (A_1)_{A_7^*}^{\perp} \,\,\,\,\, \hbox{but}\,\,\,\, k\alpha_7 \notin ((A_1)_{A_7}^{\perp})_{\hbox{root}}=A_5$$
\item
$$(A_1)_{A_9}^{\perp}=(\langle a_9\rangle)_{A_9}^{\perp} =\langle a_9+2a_8,a_7,a_6,a_5,a_4,a_3,a_2,a_1\rangle$$
$$\alpha_9 \in (A_1)_{A_9^*}^{\perp} \,\,\,\,\, \hbox{but}\,\,\,\,\, k\alpha_9 \notin ((A_1)_{A_9}^{\perp})_{\hbox{root}}=A_7$$
\item
$$(A_1)_{A_{11}}^{\perp}=(\langle a_{11} \rangle )_{A_{11}}^{\perp}=\langle a_{11}+2a_{10},a_9,a_8,a_7,a_6,a_5,a_4,a_3,a_2,a_1 \rangle$$
$$\alpha_{11}\in (A_1)_{A_{11}^*}^{\perp}\, \,\,\,\,\,\hbox{but}\,\,\,\,
 k\alpha_{11} \notin ((A_1)_{A_{11}}^{\perp})_{\hbox{root}}=A_9 $$
\item
$$(A_1)_{D_6}^{\perp} =\langle d_5 \rangle \oplus \langle d_5+d_6+2d_4+d_3,d_3,d_2,d_1 \rangle =A_1 \oplus D_4$$
$$\overline{\delta}_6 \,\,\,\,\,\hbox{and}\,\,\,\,\, \tilde{\delta}_6 \in (A_1)_{D_6^*}^{\perp},\,\,\,\,\,\delta_6 \notin (A_1)_{D_6^*}^{\perp}$$
\item
$$(A_1)_{D_7}^{\perp}=\langle d_7 \rangle_{D_7}^{\perp} =\langle d_6 \rangle \oplus \langle d_6+d_7+2d_5+d_4,d_4,d_3,d_2,d_1 \rangle=A_1\oplus D_5$$
$$3\delta_7 \in (A_1)_{D_7^*}^{\perp}$$
\item
$$(A_1)_{D_{10}}^{\perp}=\langle d_{10} \rangle_{D_{10}}^{\perp}=\langle d_9 \rangle \oplus \langle d_9+d_{10}+2d_8+d_7,d_7,d_6,d_5,d_4,d_3,d_2,d_1 \rangle =A_1 \oplus D_8$$
$$2\overline{\delta}_{10} \in A_1 \oplus D_8$$
\item
$$(A_1)_{E_6}^{\perp}=\langle e_1 \rangle _{E_6}^{\perp}=\langle e_1+e_2+2e_3+2e_4+e_5,e_6,e_5,e_4,e_2 \rangle =A_5$$
$$3\eta_6 \in A_5$$
\item
$$(A_1)_{E_7}^{\perp}=\langle e_1+e_2+2e_3+2e_4+e_5,e_7,e_6,e_5,e_4,e_2 \rangle
=D_6 $$
$$2\eta_7 \in (A_1)_{E_7}^{\perp}$$
\item $$(A_1)_{E_8}^{\perp}=E_7$$
\item $$(D_5)_{D_l}^{\perp}=D_{l-5}$$
$$D_1=(-4) \,\,\,\, D_2=A_1^{\oplus 2}\,\,\,\,D_3=\left (
                                                \begin{matrix}
                                                -2 & 0 & 1 \\
                                                 0 & -2 & 1 \\
                                                 1 & 1 & -2
                                                  \end{matrix} 
                                                 \right ) \simeq A_3$$
\item
$$(D_5)_{D_6}^{\perp}=\langle d_5+d_6+2d_4+2d_3+2d_2+2d_1 \rangle=\langle (-4)\rangle$$
$${\delta}_6 \,\,\,\,\,\hbox{and}\,\,\,\,\, \tilde{\delta}_6 \notin (D_5)_{D_6^*}^{\perp},\,\,\,\,\,\overline{\delta}_6 \in (D_5)_{D_6^*}^{\perp}$$
\item
$$(D_5)_{E_6}^{\perp}=\langle e_2,e_5,e_4,e_3,e_1 \rangle _{E_6}^{\perp}=\langle 3e_2+2e_1+4e_3+6e_4+5e_5+4e_6 \rangle =\langle (-12) \rangle$$ 
$$3\eta_6 =(-12)$$
\item
$$
\begin{array}{ll}
(D_5)_{E_7}^{\perp} &= \langle 2e_1+2e_2+3e_3+4e_4+3e_5+2e_6+e_7 \rangle \\
                   & \oplus \langle e_2+e_3+2e_4+2e_5+2e_6+2e_7 \rangle \\
                   &=A_1\oplus (-4)
\end{array}
$$
$$\eta_7\in (D_5)_{E_7^*}^{\perp},\,\,\,\,\,\,2\eta_7 \notin A_1$$
\item
$$(D_5)_{E_8}^{\perp}=A_3$$
\item 
$$(D_5 \oplus A_1 )_{D_7}^{\perp}=A_1=(d_6+d_7+2d_5+2d_4+2d_3+2d_2+d_1)$$
\item
$$
\begin{array}{ll}
(D_5\oplus A_1)_{D_8}^{\perp} & =\langle d_7+d_8+2d_6+2d_5+2d_4+2d_3+d_2\rangle \\
      & \oplus \langle d_7+d_8+2d_6+2d_5+2d_4+2d_3+2d_2+2d_1 \rangle \\
     & = A_1 \oplus (-4)
\end{array}
$$
$$4\delta_8 \notin (D_5 \oplus A_1)_{D_8}^{\perp},\,\,\,\,\,\overline{\delta}_8\notin A_1 $$
\item
$$
\begin{array}{lll}
(D_5 \oplus A_1 )_{D_9}^{\perp} & = & A_1\oplus A_1 \oplus A_1 \\
                                 &  = & (d_8+d_9+2d_7+2d_6+2d_5+2d_4+d_3)\\
  &  & \oplus (d_8+d_9+2d_7+2d_6+2d_5+2d_4+2d_3+2d_2+d_1)  \oplus (d_1)
\end{array}
$$
$$(D_5 \oplus A_1 )_{D_{10}}^{\perp}=A_1\oplus A_3 \,\,\,\,\,\,\,\,(D_5 \oplus A_1 )_{D_{12}}^{\perp}=A_1\oplus D_5$$
$$(D_5 \oplus A_1 )_{D_{16}}^{\perp}=A_1\oplus D_9 \,\,\,\,\,\,\,\,(D_5 \oplus A_1 )_{D_{24}}^{\perp}=A_1\oplus D_{17}$$
\item
$$(D_5\oplus A_1)_{E_7}^{\perp} = \langle 3e_2+2e_1+4e_3+6e_4+5e_5+4e_6+2e_7\rangle =\langle (-4)\rangle$$
$$2\eta_7 \notin (D_5 \oplus A_1)_{E_7}^{\perp}$$
\item
$$
\begin{array}{lll}
(D_5 \oplus A_1 )_{E_8}^{\perp} & = & A_1\oplus (-4) \\
                                 &  = & (3e_2+2e_1+4e_3+6e_4+5e_5+4e_6+3e_7+2e_8)\\
  &  & \oplus (3e_2+2e_1+4e_3+6e_4+5e_5+4e_6+2e_7)  
\end{array}
$$
\end{enumerate}

\end{proposition}

\begin{proof}

The orthogonal complements are given in Nishiyama \cite{Nis} and the rest of the proof follows immediately from the various expressions of $\eta_7$, $\eta_6$, $\delta_l$, $\overline{\delta}_l$, $\tilde{\delta}$ and $\alpha_m$ given in section 4.

\end{proof}

Once the different types of fibrations are known, we get the rank of the Mordell-Weil group by 5.3.2.

To determine the torsion part we need to know appropriate generators of  $L/L_{\hbox{root}}$. 

\subsection{Generators of $L/L_{\hbox{root}}$}
By lemma 1, a set of generators can be described in terms of elements of $L^*/L$. We list in the following table the generators fitting to the corresponding $W$. We restrict to relevant $L_{\hbox{root}}$ according to proposition 2.

For convenience, the generators are given modulo $L_{\hbox{root}}$.

\begin{center}
\begin{tabular}{|l|c|}
\hline
   
\rule{0cm}{0.5cm}$L_{\hbox{root}}$ & $L/L_{\hbox{root}}$ \\ \hline
   
\rule{0cm}{0.5cm} $E_8^3$    &  $\langle(0) \rangle$\\ \hline
  
\rule{0cm}{0.5cm} $E_8 D_{16}$    &  $\langle \delta_{16} \rangle \simeq \mathbb Z /{2 \mathbb Z}$ \\ \hline
  
\rule{0cm}{0.5cm} $E_7^2 D_{10}$    &  $\langle \eta_7^{(1)}+\delta_{10}, \eta_7^{(2)}+\bar{\delta_{10}} \rangle \simeq (\mathbb Z /{2 \mathbb Z})^2$ \\ \hline
   
\rule{0cm}{0.5cm} $E_7 A_{17}$    &  $\langle \eta_7+3\alpha_{17} \rangle \simeq \mathbb Z /{6 \mathbb Z}$ \\ \hline
   
\rule{0cm}{0.5cm} $D_8^3$    &  $\langle \delta_8^{(1)}+\overline{\delta}_8^{(2)}+\overline{\delta}_8^{(3)}, \overline{\delta}_8^{(1)}+\delta_8^{(2)}+\overline{\delta}_8^{(3)},\overline{\delta}_8^{(1)}+\overline{\delta}_8^{(2)}+\delta_8^{(3)}\rangle \simeq (\mathbb Z /{2 \mathbb Z})^3$ \\ \hline
    
\rule{0cm}{0.5cm} $D_9 A_{15}$    &  $\langle \delta_9+2\alpha_{15} \rangle \simeq \mathbb Z /{8 \mathbb Z}$ \\ \hline
     
\rule{0cm}{0.5cm} $E_6^4$    &  $\langle \eta_6^{(1)}+\eta_6^{(2)}+\eta_6^{(3)}, 2\eta_6^{(1)}+\eta_6^{(3)}+\eta_6^{(4)} \rangle \simeq (\mathbb Z /{3 \mathbb Z)^2}$\\ \hline
    
\rule{0cm}{0.5cm} $ A_{11} E_6 D_7$    &  $ \langle \alpha_{11}+\eta_6+\delta_7 \rangle \simeq \mathbb Z /{12 \mathbb Z}$ \\ \hline
    
\rule{0cm}{0.5cm} $D_6^4$    &  $\langle \overline{\delta}_6^{(1)}+\overline{\delta}_6^{(4)},\overline{\delta}_6^{(2)}+\overline{\delta}_6^{(3)},\delta_6^{(1)}+\overline{\delta}_6^{(3)}+\delta_6^{(4)},\overline{\delta}_6^{(1)}+\overline{\delta}_6^{(2)}\rangle \simeq (\mathbb Z/2\mathbb Z)^4 $\\ \hline
    
\rule{0cm}{0.5cm} $D_6 A_9^2$    &  $\langle \delta_6+5\alpha_9^{(2)}, \delta_6+\alpha_9^{(1)}+2\alpha_9^{(2)} \rangle \simeq \mathbb Z /{2 \mathbb Z}\times \mathbb Z /{10 \mathbb Z} $ \\ \hline
    
\rule{0cm}{0.5cm} $D_5^2 A_7^2$    &  $\langle \delta_5^{(1)}+\delta_5^{(2)}+2\alpha_7^{(1)},\delta_5^{(1)}+2\delta_5^{(2)}+\alpha_7^{(1)}+\alpha_7^{(2)} \rangle \simeq \mathbb Z /{4 \mathbb Z}\times \mathbb Z /{8 \mathbb Z} $ \\ \hline

\end{tabular}
\end{center}

If the rank is $0$, we apply lemma $4$ ($1$) and $(2)$ to determine the torsion part of the Mordell-Weil group. Thus we recover the $14$ fibrations of rank $0$ exhibited by Shimada and Zhang \cite{Shim}.

For the other $16$ fibrations we apply proposition 5 and lemma 4 (3).

Recall that $\det  W =8$ and the torsion group is $\overline{W}_{\hbox{root}}/W_{\hbox{root}}$.

\subsection{$L_{\hbox{root}}=E_8D_{16}$}
 $$L/L_{\hbox{root}}=\langle \delta_{16} +L_{\hbox{root}}\rangle \simeq \mathbb Z /{2 \mathbb Z}$$

\subsubsection{Fibration $A_1D_{16}$}
It is obtained from the primitive embedding $D_5\oplus A_1\subset E_8$.
Since by proposition 5 (21)$(D_5\oplus A_1)_{E_8}^{\perp}=A_1\oplus (-4)$, $\hbox{det}\,N=8\times 4$, so $W/N \simeq \mathbb Z /2\mathbb Z \simeq L/L_{\hbox{root}}=\langle \delta_{16}+L_{\hbox{root}} \rangle$. Since $2\delta_{16}\in D_{16} =W_{\hbox{root}}$, thus $\delta_{16} \in \overline{W}_{\hbox{root}}$ and
$$\overline{W}_{\hbox{root}}/W_{\hbox{root}}= \mathbb Z /2\mathbb Z.$$

\subsection{$L_{\hbox{root}}=E_7^2D_{10}$}
 $$L/L_{\hbox{root}}=\langle \eta_7^{(1)}+\delta_{10}, \eta_7^{(2)}+\bar{\delta}_{10}\,\,{\hbox{mod.}}\,\,L_{\hbox{root}} \rangle \simeq (\mathbb Z /{2 \mathbb Z})^2 $$

\subsubsection {Fibration $A_1D_6D_{10}$}
It is obtained from the primitive embeddings $A_1 \subset E_7^{(1)}$ and $D_5 \subset E_7^{(2)}$.
Since by proposition 5 (10) and (15) $(A_1)_{E_7^{(1)}}^{\perp}=D_6$ and $(D_5)_{E_7^{(2)}}^{\perp}=A_1 \oplus (-4)$, we get $\hbox{det}\,N=8\times 4^2$ and $W/N\simeq (\mathbb Z/2 \mathbb Z )^2\simeq L/L_{\hbox{root}}$.

By proposition 5 (15), $\eta_7^{(1)} \in (D_5)_{E_7^{(1)*}}^{\perp}=A_1 \oplus (-4)$, but $2\eta_7 \notin A_1$ and by proposition 5 (10) $2\eta_7^{(1)} \in (A_1)_{E_7^{(1)}}^{\perp}=D_6$. So
$$\overline{W}_{\hbox{root}}/W_{\hbox{root}}=\langle \eta_7^{(1)}+\delta_{10}+W_{\hbox{root}} \rangle\simeq \mathbb Z /2\mathbb Z.$$

\subsubsection{Fibration $E_7D_{10}$}

It is obtained from $D_5\oplus A_1 \subset E_7^{(1)}$.

Since by proposition 5 (20) $(D_5\oplus A_1)_{E_7}^{\perp}=(-4)$, $\det N=8\times 4$ so $W/N \simeq \mathbb Z /2 \mathbb Z$.

Again by proposition 5 (20), $2\eta_7^{(1)} \notin (D_5\oplus A_1)_{E_7^{(1)}}^{\perp}$ and we get $W/N=\langle \eta_7^{(2)}+\bar{\delta}_{10}+N \rangle$. Since $2\eta_7^{(2)} \in E_7$ and $2\bar{\delta}_{10} \in D_{10}$, it follows

$$\overline{W}_{\hbox{root}}/W_{\hbox{root}}\simeq \mathbb Z /2\mathbb Z.$$

\subsubsection{Fibration $2A_1D_8E_7$}

It is obtained from the primitive embeddings $A_1\subset D_{10}$ and $D_5 \subset E_7^{(1)}$.

By proposition 5 (8) and (15), we get
$(A_1)_{D_{10}}^{\perp}=A_1 \oplus D_{8}$ and $(D_5)_{E_7^{(1)}}^{\perp}=A_1 \oplus (-4)$, so $\hbox{det}\, N=8\times 4^2$ and 
 $W/N\simeq (\mathbb Z/2 \mathbb Z )^2 \simeq L/L_{\hbox{root}}$.

By proposition 5 (15), $2\eta_7^{(1)}\notin A_1$ and by proposition 5 (8) $2\bar{\delta}_{10}\in A_1\oplus D_8$ so

$$\overline{W}_{\hbox{root}}/W_{\hbox{root}}\simeq \mathbb Z /2\mathbb Z.$$

\subsection{$L_{\hbox{root}}=E_7A_{17}$}

$$L/L_{\hbox{root}}=\langle \eta_7+3\alpha_{17}+L_{\hbox{root}} \rangle \simeq \mathbb Z/6 \mathbb Z$$

\subsubsection{Fibration $A_{17}$}
It is obtained from the primitive embedding $D_5\oplus A_1 \subset E_7$.

By proposition 5 (20), $(D_5\oplus A_1)_{E_7}^{\perp}=(-4)$, so $\hbox{det}\,N=8\times 9$ and $W/N \simeq \mathbb Z/3 \mathbb Z=\langle 6\alpha_{17} +N \rangle$.

Moreover, since $18\alpha_{17}\in A_{17}$, $6\alpha_{17}\in \overline{W}_{\hbox{root}}$ so

$$\overline{W}_{\hbox{root}}/W_{\hbox{root}}\simeq \mathbb Z /3\mathbb Z.$$

\subsubsection{Fibration $A_1A_{15}$}
It is obtained from the primitive embeddings $D_5 \subset E_7$ and $A_1 \subset A_{17}$.
By proposition 5 (15) and (1), $(D_5)_{E_7}^{\perp}=A_1\oplus (-4)$ and $(A_1)_{A_{17}}^{\perp}=L_{15}^2$ with $\det L_{15}^2=2\times 18$, so $\hbox{det}\,N=8\times 6^2$ and $W/N\simeq \mathbb Z/6 \mathbb Z \simeq L/L_{\hbox{root}}$.

But, by lemma 1, $W_{\hbox{root}}=A_1A_{15}$ has no overlattice. Hence

$$\overline{W}_{\hbox{root}}/W_{\hbox{root}}\simeq (0).$$

\subsection{$L_{\hbox{root}}=D_8^3$}

$$
\begin{array}{ll}
L/L_{\hbox{root}} & =\langle \delta_8^{(1)}+\overline {\delta}_8^{(2)}+\overline{\delta}_8^{(3)}, \overline{\delta}_8^{(1)}+ {\delta}_8^{(2)}+\overline{\delta}_8^{(3)},\overline{ \delta}_8^{(1)}+\overline {\delta}_8^{(2)}+{\delta}_8^{(3)}\rangle\,\, \hbox{mod.}\,\, L_{\hbox{root}}\\
       &   \simeq (\mathbb Z/2 \mathbb Z)^3
\end{array}
$$

\subsubsection{ Fibration $A_1 D_8 D_8$}

It comes from the primitive embedding $D_5\oplus A_1 \subset D_8^{(1)}$.

By proposition 5 (18), $(D_5\oplus A_1)_{D_8}^{\perp}=A_1\oplus (-4)$ so $\hbox{det}\,N=8\times 4^2$ and $W/N\simeq (\mathbb Z/2\mathbb Z)^2$.

By proposition 5 (18), $4\delta_8 \notin (D_5\oplus A_1)_{D_8}^{\perp}$ so $W/N=\langle  \overline{\delta}_8^{(1)}+ {\delta}_8^{(2)}+\overline{\delta}_8^{(3)}, \overline{\delta}_8^{(1)}+\overline {\delta}_8^{(2)}+{\delta}_8^{(3)}\rangle $. Again by proposition 5 (18) $\overline{\delta}_8 \notin A_1$, so only $2( \delta_8^{(2)}+\overline {\delta}_8^{(2)}+\overline{\delta}_8^{(3)}+\delta_8^{(3)})\in W_{\hbox{root}}$ and

$$\overline{W}_{\hbox{root}}/W_{\hbox{root}}\simeq \mathbb Z /2\mathbb Z.$$

\subsection{$L_{\hbox{root}}=D_9A_{15}$}

$$L/L_{\hbox{root}}=\langle \delta_9+2\alpha_{15}+L_{\hbox{root}} \rangle \simeq \mathbb Z/8 \mathbb Z$$

\subsubsection{ Fibration $D_4A_{13}$}
It comes from the primitive embeddings $D_5 \subset D_9$ and $A_1 \subset A_{15}$.

By proposition 5 (12) and (1), $(D_5)_{D_9}^{\perp}=D_4$, $(A_1)_{A_{15}}^{\perp}=L_{13}^2$ with $\hbox{det}\, L_{13}^2=2\times 16$ so $\hbox{det}\,N=8\times 16$ and $W/N \simeq \mathbb Z/4 \mathbb Z$.

But by lemma 1, $W_{\hbox{root}}=D_4A_{13}$ has no overlattice since $q_{A_{13}}(\alpha_{13})=\left (-\frac {1}{14} \right )$ and $q_{D_4}(\delta_4) \in \mathbb Z$. Hence

$$\overline{W}_{\hbox{root}}/W_{\hbox{root}}\simeq (0).$$

\subsection{$L_{\hbox{root}}=E_6^4$}

$$L/L_{\hbox{root}}=\langle \eta_6^{(1)}+\eta_6^{(2)}+\eta_6^{(3)}, 2\eta_6^{(1)}+\eta_6^{(3)}+\eta_6^{(4)} \rangle \,\,\hbox{mod.} L_{\hbox{root}} \,\,\simeq (\mathbb Z/3 \mathbb Z)^2$$

\subsubsection{Fibration $A_5 E_6 E_6$}

We can suppose the primitive embeddings $A_1 \subset E_6^{(1)}$ and $D_5 \subset E_6^{(2)}$.

By proposition 5 (9) and (14), $(A_1)_{E_6}^{\perp}=A_5$ and $(D_5)_{E_6}^{\perp}=(-12)$ so $\hbox{det}\,N=8\times 9^2$ and $W/N=(\mathbb Z/3 \mathbb Z)^2$.

By proposition 5 (9), $3\eta_6^{(1)} \in A_5$  so $2\eta_6^{(1)}+\eta_6^{(3)}+\eta_6^{(4)} \in \overline{W}_{\hbox{root}}$ but $\eta_6^{(1)}+\eta_6^{(2)}+\eta_6^{(3)} \notin \overline{W}_{\hbox{root}}$ by proposition 5 (14). Hence

$$\overline{W}_{\hbox{root}}/W_{\hbox{root}}\simeq \mathbb Z /3\mathbb Z.$$

\subsection{$L_{\hbox{root}}=A_{11} E_6 D_7$}
$$L/L_{\hbox{root}}= \langle \alpha_{11}+\eta_6+\delta_7+L_{\hbox{root}} \rangle \simeq \mathbb Z /12 \mathbb Z$$

\subsubsection{ Fibration $A_9 A_1 A_1 E_6$}
It follows from the primitive embeddings $A_1 \subset A_{11}$ and $D_5 \subset D_7$.

By proposition 5 (5) and (12), $(A_1)_{A_{11}}^{\perp}=L_9^2$, $\det L_9^2=2\times 12$, $(D_5)_{D_7}^{\perp}=D_2 \simeq A_1^{\oplus 2}$ so $\det N=8\times 6^2$ and $W/N \simeq \mathbb Z/6 \mathbb Z =\langle 2\alpha_{11}+2\eta_6+2\delta_7+N \rangle $.

Since $k(2\alpha_{11}) \notin A_9$ by proposition 5 (5), we get

$$\overline{W}_{\hbox{root}}/W_{\hbox{root}}= (0).$$

\subsubsection{Fibration $A_9 D_7$}

By lemma 1, it follows that $W_{\hbox{root}}=A_9 D_7$ has no overlattice since $q_{A_9}(\alpha _{9})=\left (-\frac {1}{10}\right )$ and $q_{D_7}=\left (-\frac {7}{4}\right )$. Hence

$$\overline{W}_{\hbox{root}}/W_{\hbox{root}}= (0).$$

\subsubsection{Fibration $A_{11} A_1 D_5$}
It comes from the primitive embeddings $D_5 \subset E_6$ and $A_1 \subset D_7$.

By proposition 5 (14) and (7), $(D_5)_{E_6}^{\perp} =(-12)$ and $(A_1)_{D_7}^{\perp}=A_1 \oplus D_5$, so $\det N=8\times 12^2$ and $W/N \simeq \mathbb Z / 12 \mathbb Z \simeq L/L_{\hbox{root}}$.

Since $W_{\hbox{root}}\cap E_6 = \emptyset $, we get also $\overline{W}_{\hbox{root}} \cap E_6 =\emptyset$.

Now $3\alpha_{11}+3\delta_7\in \overline{W}_{\hbox{root}}$ since $3\delta_7 \equiv \tilde{\delta}_7$ and $4(3\alpha_{11}+3\delta_7)\in W_{\hbox{root}}$. Hence

$$\overline{W}_{\hbox{root}}/W_{\hbox{root}}\simeq \mathbb Z /4\mathbb Z.$$

\subsection{$L_{\hbox{root}}=D_6^4$}

$$L/L_{\hbox{root}}=\langle \overline{\delta}_6^{(1)}+\overline{\delta}_6^{(4)},\overline{\delta}_6^{(2)}+\overline{\delta}_6^{(3)},\delta_6^{(1)}+\overline{\delta}_6^{(3)}+\delta_6^{(4)},\overline{\delta}_6^{(1)}+\overline{\delta}_6^{(2)}\rangle\,\, {\hbox{mod.}}\,\, L_{\hbox{root}} \simeq (\mathbb Z/2\mathbb Z)^4$$

\subsubsection{Fibration $A_1D_4D_6D_6$}

We assume the primitive embeddings $A_1 \subset D_6^{(1)}$ and $D_5 \subset D_6^{(2)}$.

By proposition 5 (13) and (6), $(A_1)_{D_6}^{\perp}=A_1\oplus D_4$, $(D_5)_{D_6}^{\perp}=(-4)$, so $\det N=8\times 8^2$ and $W/N\simeq (\mathbb Z/2 \mathbb Z)^3$.

After enumeration of all the elements of $L/L_{\hbox{root}}$, since by proposition 5 (13) and (6) $\overline{\delta}_6^{(2)}\in (D_5)_{D_6^*}^{\perp}$ and only $\overline{\delta}_6^{(1)}$ or $\tilde{\delta}_6^{(1)}\in(A_1)_{D_6^*}^{\perp}$, we get 

$$W/N=$$
$$\{\overline{\delta}_6^{(1)}+\overline{\delta}_6^{(4)},\overline{\delta}_6^{(2)}+\overline{\delta}_6^{(3)},\overline{\delta}_6^{(1)}+\overline{\delta}_6^{(2)},\overline{\delta}_6^{(2)}+\overline{\delta}_6^{(4)},\overline{\delta}_6^{(1)}+\overline{\delta}_6^{(3)},\overline{\delta}_6^{(3)}+\overline{\delta}_6^{(4)},\overline{\delta}_6^{(1)}+\overline{\delta}_6^{(2)}+\overline{\delta}_6^{(3)}+\overline{\delta}_6^{(4)}\}$$
$$\simeq (\mathbb Z/2 \mathbb Z)^3$$

As $2\overline{\delta}_6^{(2)}\notin W_{\hbox{root}}$ and $2\overline{\delta}_6^{(1)}\in A_1\oplus D_4$, it follows

$$\overline{W}_{\hbox{root}}/W_{\hbox{root}}=\{\overline{\delta}_6^{(1)}+\overline{\delta}_6^{(4)},\overline{\delta}_6^{(1)}+\overline{\delta}_6^{(3)},\overline{\delta}_6^{(3)}+\overline{\delta}_6^{(4)},0 \}$$
$$\simeq \mathbb Z/2 \mathbb Z \times \mathbb Z/2\mathbb Z $$

\subsection{$L_{\hbox{root}}=D_6A_9^2$}

$$L/L_{\hbox{root}}=\langle \delta_6+5\alpha_9^{(2)},\delta_6+\alpha_9^{(1)}+2\alpha_9^{(2)} \rangle\,\, \hbox{mod.} L_{\hbox{root}}\,\,\simeq \mathbb Z /2 \mathbb Z \oplus \mathbb Z/10 \mathbb Z $$

\subsubsection{Fibration $A_7A_9$}

We assume the following primitive embeddings $D_5\subset D_6$ and $A_1 \subset A_9^{(1)}$.

By proposition 5 (13) and (4), $(D_5)_{D_6}^{\perp}=(-4)$, $(A_1)_{A_9}^{\perp}=L_7^2$, $\det L_7^2=2\times 10$ so $\det N=8\times 10^2$ and $[W:N]=10$.

Enumerating the elements of $L/L_{\hbox{root}}$ and since $\delta_6 \notin (D_5)_{D_6^*}^{\perp}$ by proposition 5 (13), we get

$$
\begin{array}{lll}
W/N & = &\{ \alpha_9^{(1)}+7\alpha_9^{(2)}, 2\alpha_9^{(1)}+4\alpha_9^{(2)}, 3 \alpha_9^{(1)}+\alpha_9^{(2)},  4\alpha_9^{(1)}+8\alpha_9^{(2)},  5\alpha_9^{(1)}+5\alpha_9^{(2)}, \\
& &  6\alpha_9^{(1)}+2\alpha_9^{(2)}, 7 \alpha_9^{(1)}+9\alpha_9^{(2)}, 8 \alpha_9^{(1)}+6\alpha_9^{(2)}, 9 \alpha_9^{(1)}+3\alpha_9^{(2)}, 0 \}\\
 & \simeq & \mathbb Z /10 \mathbb Z
\end{array}
$$
Since $k\alpha_9^{(1)} \notin A_7$ by proposition 5 (4), it follows

$$\overline{W}_{\hbox{root}}/W_{\hbox{root}}=(0).$$

\subsection{$L_{\hbox{root}}=D_5^2A_7^2$}

$$L/L_{\hbox{root}}=\langle 2\alpha_7^{(1)}+\delta_5^{(1)}+\delta_5^{(2)}, \alpha_7^{(1)}+\alpha_7^{(2)}+\delta_5^{(1)}+2\delta_5^{(2)} \rangle \hbox{mod.}\,\, L_{\hbox{root}} \simeq \mathbb Z/4 \mathbb Z \times \mathbb Z /8 \mathbb Z$$

\subsubsection{ Fibration $D_5A_5A_7$}

We assume the following primitive embeddings $D_5 \subset D_5^{(2)}$ and $A_1 \subset A_7^{(1)}$.

By proposition 5 (3), $(A_1)_{A_7}^{\perp}=L_5^2$, $\det L_5^2=2\times 8$ so $\det N=8\times 8^2$ and $[W:N]=8$.

Now enumerating the elements of $L/L_{\hbox{root}}$ and since $\delta_5^{(2)}$ does not occur in $W$, we get $W/N=\langle 5\alpha_7^{(1)}+\alpha_7^{(2)}+3\delta_5^{(1)}+N \rangle$. Since by proposition 5 (3) $k\alpha_7^{(1)} \notin A_5$, it follows

$$\overline{W}_{\hbox{root}}/W_{\hbox{root}}=(0).$$

\end{proof}

\section{Equations of fibrations}

In the next sections we give  Weierstrass equations of all the
elliptic fibrations. We will use the following proposition (\cite{P} p.559-560 or \cite{SS} Prop. 12.10) .
\begin{proposition}
Let $X$ be a $K3$ surface and $D$ an effective divisor  on $X$ that has the same type as a singular fiber of an elliptic fibration.
Then X admits a unique elliptic fibration with D as a singular fiber. Moreover,
any irreducible curve $C$ on $X$ with $D.C = 1$ induces a section of the elliptic
fibration.
\end{proposition}
First we show that one of the fibrations is the modular
elliptic surface with base curve the modular curve $X_1(8)$ corresponding to modular group  $\Gamma_{1}(8).$ As we see in the Table 1,  it
corresponds to the fibration $A_{1},A_{3},2A_{7}.$ The Mordell-group is a torsion group of order $8$. We determine a graph with the singular fibers $I_{2},I_{4},2I_{8}$ and the $8$-torsion
sections. Most divisors used in the previous proposition can be drawn on the graph.

 From this modular fibration we can  easily write a Weierstrass equation of two other  fibrations of parameters $k$ and $v$. From the singular fibers of these two fibrations we obtain the divisors of a set of functions on $Y_2$. These
functions generate a group whose horizontal divisors correspond to $8$-torsion
sections. These divisors lead to more fibrations.

If $X$ is a $K3$ surface and 
\[
\pi : X \rightarrow C
\]
an elliptic fibration, then the curve $C$ is of genus $0$ and we define an \textbf  {elliptic parameter} as a generator of the function field of $C.$ The parameter is not unique but defined up to  linear fractional transformations. 

From the previous proposition  we can obtain equations from the linear system of $D$.
Moreover if we have two effective divisors $D_{1}$ and $D_{2}$ for the same
fibration we can choose an elliptic parameter with divisor $D_{1}-D_{2}.$ We
give all the details for the fibrations of parameter $t$ and $\psi.$

For each elliptic fibration we will give a Weierstrass model numbered from $1$ to $30$, generally in the two variables $y$ and $x$. Parameters are denoted with small latine or greek letters. In most cases we give the change of variables that converts the defining equation into a Weierstrass form.
 Otherwise we use standard algorithms to obtain a Weierstrass form (see for example
\cite{Ca}). From a Weierstrass equation we get the singular fibers, using \cite{T} for example,  thus  the corresponding fibration in Table 1; so we know the rank and the torsion of the Mordell-Weil group. 
If the rank is $>0$ we give points and heights of points, which, using the formula of Proposition $1$,
generate the Mordell-Weil lattice. Heights are computed with Weierstrass equations as explained in \cite{Ku}.  
Alternatively we can compute heights as in \cite{Shio} and \cite{HL}.

\subsection{Equation of the modular surface associated to the modular group $\Gamma_{1}(8)$}

We start with the elliptic surface
\[
X+\frac {1}{X}+Y+\frac {1}{Y}=k.
\]

From Beauville's  classification \cite{Bea}, we know that it is the modular elliptic surface
corresponding to the modular group $\Gamma_{1}(4)\cap\Gamma_{0}(8).$ Using the birational
transformation
\[
X=\frac{-U\left(  U-1\right)  }{V}\text{ and }Y=\frac{V}{U-1}
\]
with inverse%
\[
U=-XY\text{ and }V=-Y\left(  XY+1\right)
\]
we obtain the Weierstrass equation
\[
V^{2}-kUV=U(U-1)^{2}.%
\]
The point $Q=\left(  U=1,V=0\right)  $ is a $4-$torsion point. If we want 
$A$ with $2A=Q$ to be a rational point, then $k=-s-1/s+2.$ It follows 
\begin{equation}
V^2+(s+\frac{1}{s}-2)UV=U(U-1)^2.
\label{s}
\end{equation}
and 
\[
X+\frac{1}{X}+Y+\frac{1}{Y}+s+\frac{1}{s}=2
\]
The point
$A=(U=s,V=-1+s)$ is of order $8.$ We obtain easily its multiples 
 {\small
\[%
\begin{array}
[c]{cccccccc}%
& A & 2A & 3A & 4A & 5A & 6A & 7A\\
\left(  X,Y\right)   & \left(  -s,1\right)   & \left(  \infty,0\right)   &
\left(  1,\frac{-1}{s}\right)   & \left(  0,0\right)   & \left(  \frac{-1}%
{s},1\right)   & \left(  0,\infty\right)   & \left(  1,-s\right)
\end{array}
\]
}

Thus we get an equation for the modular surface $Y_2$  associated to the modular group
$\Gamma_{1}(8)$ 
\[
Y_2:X+\frac{1}{X}+Y+\frac{1}{Y}+Z+\frac{1}{Z}-2=0.
\]
and the elliptic fibration 
\[
\left(  X,Y,Z\right)     \mapsto Z=s.
\]


Its singular fibers  are
\[%
\begin{array}
[c]{cccc}%
\text{at} & s=0 & \text{of type} & I_{8}\\
\text{at} & s=\infty & \text{of type} & I_{8}\\
\text{at} & s=1 & \text{ of type } & I_{4}\\
\text{at} & s=-1 & \text{of type} & I_{2}\\
\text{at} & s=3+2\sqrt{2} & \text{of type} & I_{1}\\
\text{at} & s=3-2\sqrt{2} & \text{of type} & I_{1}%
\end{array}
\]

\subsection{Construction of the graph from the modular fibration}

At $s=s_{0}$, we have a singular fiber of type $I_{n_{0}}.$ We denote
$\Theta_{s_{0},j}$ with $s_{0}\in\{0,\infty,1,-1\}$, $j\in\{0, \ldots ,n_{0}-1\}$ the
components of a singular fiber $I_{n_{0}}$ such that  $\Theta_{i,j}.$
$\Theta_{k,j}=0$ if $i\neq k$ and%
\[
\Theta_{i,j}.\Theta_{i,k}=\left\{
\begin{array}
[c]{l}%
1\text{ if }|k-j|=1 \,\text{or}\, |k-j|=n_0-1\\
-2\text{ if }k=j\\
0\text{ otherwise }
\end{array}
\right.
\]

 the dot meaning the intersection product. By definition, the component
$\Theta_{k,0}$ is met by the zero section $(0)$. The $n_0$-gone obtained can
be oriented in two ways for $n_0>2$. For each $s_{0}$ we want to know which component is cut by the section $(A)$, i.e.  the index
$j(A,s_{0})$ such that $A.\Theta_{s_{0},j\left(  A,s_{0}\right)  }=1.$ For
this, we compute the local height for the prime $s-s_{0}$ with the
Weierstrass equation. Since this height is also equal to $\frac{j(A,s_{0})\left(
n_{s_{0}}-j(A,s_{0})\right)  }{n_{s_{0}}}$ we can give an
orientation to the $n_{0}$-gone by choosing $0\leq$ $j(A,s_{0})\leq
\frac{n_{s_{0}}}{2}.$ Hence we get the following results: $j(A,0)=3$, $j(A,s_0)=1$ for $s_0 \neq 0.$


For  the other torsion-sections $\left(  iA\right)$ we use the algebraic
structure of the N\'{e}ron model and get $\left(  iA\right)  .\Theta_{0,j}=1$ if
$j=3i \mod 8$, $ \left(  iA\right)  .\Theta_{0,j}=0$ if $j\neq3i.$ For $s_{0}\in\{\infty,1,-1\}$ we have
$\left(  iA\right)  .\Theta_{s_{0},j}=1$ if $i=j \mod n_{0}.$
\begin{remark}
We can also compute $j(A,s_0)$ explicitly from equation of section using a new Weierstrass equation fitting with N\'{e}ron theorem (\cite{Ne} Theorem 1 and  prop. 5 p 96).
\end{remark}

Now we can draw the following graph. The vertices are the sections $\left(
iA\right) $ and the components $\Theta_{s_{0},i}$ with $s_{0}\in\{0,\infty
,1,-1\},$  $j\in\{0,1..n_{0}\}.$ Two vertices $B$ and $C$ are linked by an
edge if $B.C=1.$ 
For simplicity the two vertices $\Theta_{-1,0}$, $\Theta_{-1,1}$ and the edge between them  are not represented. The edges joining $\Theta_{-1,0}$ and $(jA)$, $j$ even  are suggested by a small segment from $(jA)$, and also edges from $\Theta_{-1,1}$ to $(iA)$, $i$ odd.



\begin{center}
\begin{tikzpicture}[scale=1.15]
\draw (3.535,3.535)-- (0,5);%
\draw (0,5)--(-3.535,3.535) ;%
\draw  (-3.535,3.535)--(-5,0);
\draw  (-5,0)--(-3.535,-3.535);
\draw  (-3.535,-3.535)--(0,-5);
\draw  (0,-5)--(3.5353,-3.5353);
\draw  (3.5353,-3.5353)--(5,0);
\draw  (5,0)--(3.535,3.535); 
\draw [fill=black](0,5) circle (0.05cm)node [above] {$\Theta_{\infty,2}$};
\draw [fill=black](-3.535,3.535) circle (0.05cm)node [above] {$\Theta_{\infty,3}$};
\draw [fill=black](3.535,3.535) circle (0.05cm)node [above] {$\Theta_{\infty,1}$};
\draw [fill=black](-5,0) circle (0.05cm)node [above] {$\Theta_{\infty,4}$};
\draw [fill=black](-3.535,-3.535) circle (0.05cm)node [below] {$\Theta_{\infty,5}$};
\draw [fill=black](0,-5) circle (0.05cm)node [left] {$\Theta_{\infty,6}$};
\draw [fill=black](3.5353,-3.5353) circle (0.05cm)node [right] {$\Theta_{\infty,7}$};
\draw [fill=black](5,0) circle (0.05cm)node [above] {$\Theta_{\infty,0}$};
\draw (0,3)--(-2.121,-2.121);
\draw (-2.121,-2.121)--(3,0);
\draw (3,0)--(-2.121,2.121);
\draw (-2.121,2.121)--(0,-3);
\draw (0,-3)--(2.121,2.121);
\draw (2.121,2.121)--(-3,0);
\draw (-3,0)--(2.121,-2.121);
\draw (2.121,-2.121)--(0,3);
\draw [fill=black](0,3)  circle (0.05cm)node [left] {$\Theta_{0,6}$};
\draw [fill=black] (-2.121,-2.121) circle (0.05cm)node[below] {$\Theta_{0,7}$};
\draw [fill=black] (3,0) circle (0.05cm)node[below] {$\Theta_{0,0}$};
\draw [fill=black] (-2.121,2.121) circle (0.05cm)node[right] {$\Theta_{0,1}$};
\draw [fill=black] (0,-3) circle (0.05cm)node[left] {$\Theta_{0,2}$};
\draw [fill=black]  (2.121,2.121)circle (0.05cm)node[right] {$\Theta_{0,3}$};
\draw [fill=black] (-3,0) circle (0.05cm)node[below] {$\Theta_{0,4}$};
\draw [fill=black]  (2.121,-2.121)circle (0.05cm)node[right] {$\Theta_{0,5}$};

\draw [fill=black](2.828,2.828) circle (0.05cm)node [below] {$A$};
\draw [red] (2.828,2.828)--(2.828,3.228);
\draw [fill=black](0,4)circle (0.05cm)node [below] {$2A$};
\draw [red] (0,4)--(0.4,4.4);
\draw [fill=black](-2.828,2.828)circle (0.05cm)node [below] {$3A$};
\draw [red] (-2.828,2.828)--(-2.828,3.228);
\draw [fill=black] (-4,0)circle (0.05cm)node [below] {$4A$};
\draw [red] (-4,0)--(-3.6,0.4);
\draw [fill=black](-2.828,-2.828) circle (0.05cm)node [below] {$5A$};
\draw[red](-2.828,-2.828)--(-2.828,-2.428);
\draw [fill=black](0,-4)circle (0.05cm)node [below] {$6A$};
\draw[red](0,-4)--(0.4,-3.6);
\draw [fill=black](2.828,-2.828)circle (0.05cm)node [below] {$7A$};
\draw[red](2.828,-2.828)--(2.828,-2.428);
\draw [fill=black](4,0) circle (0.05cm)node [below] {$0$};
\draw [red] (4.4,0.4)--(4,0);

\draw [red]  (0,4)--(0,3);
\draw [red]   (-2.828,2.828)--(-2.121,2.121);
\draw [red]   (4,0)--(3,0);
\draw [red]   (-2.828,-2.828)--(-2.121,-2.121);
\draw [red]   (2.828,-2.828)--(2.121,-2.121);
\draw [red]   (2.828,2.828)--(2.121,2.121);
\draw [red]   (-4,0)--(-3,0);
\draw [red]   (0,-4)--(0,-3);

\draw [red]   (0,4)--(0,5);%
\draw [red]   (-2.828,2.828)--(-3.535,3.535);
\draw [red]   (4,0)--(5,0);
\draw [red]   (-2.828,-2.828)--(-3.535,-3.535);
\draw [red]   (2.828,-2.828)--(3.5353,-3.5353);
\draw [red]   (2.828,2.828)--(3.535,3.535);
\draw [red]   (-4,-0)--(-5,0);
\draw [red]   (0,-4)--(0,-5);

\draw [fill=black](0.8,0) circle (0.05cm)node [below] {$\Theta_{1,2}$};
\draw [fill=black](0,0.8) circle (0.05cm)node [below] {$\Theta_{1,0}$};
\draw [fill=black](-0.8,0) circle (0.05cm)node [below] {$\Theta_{1,1}$};
\draw [fill=black](0,-0.8) circle (0.05cm)node [below] {$\Theta_{1,3}$};

\draw (0,0.8)--(-0.8,0);
\draw  (-0.8,0)--(0.8,0);
\draw (0.8,0)--(0,-0.8);
\draw (0,-0.8)--(0,0.8);

\draw [red] (4,0)--(0,0.8);
\draw[red](-4,0)--(0,0.8);
\draw[red](2.828,2.828)--(-0.8,0);
\draw [red](-2.828,-2.828)--(-0.8,0);
\draw[red](0,4)--(0.8,0);
\draw[red](0,-4)--(0.8,0);
\draw[red](-2.828,2.828)--(0,-0.8);
\draw[red](2.828,-2.828)--(0,-0.8);

\end{tikzpicture}%
\end{center} 

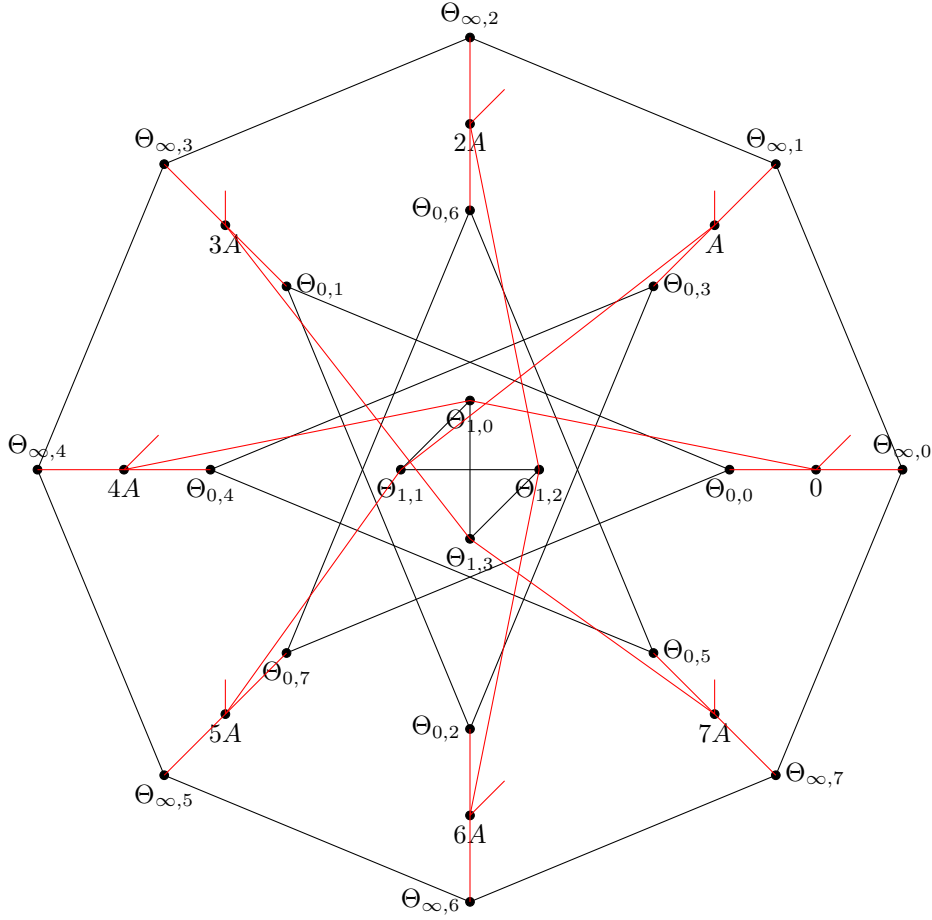
\captionof{figure}{Graph of bad fibers for $s=0,\infty,1,-1$ and torsion-sections}
\label{fig: G}

\subsection{Two fibrations}
For the two fibrations to be considered we use the following factorizations of the equation of the surface:

\[
(X+Y)(XY+1)Z+XY(Z-1)^2=0
\]
\[
(X+ZY)(XZ+Y)+(X+Y)(Y-1)(X-1)Z=0
\]
\subsubsection{Fibration of parameter $k$}
The parameter of the first one is $k=X+Y.$
Eliminating for example $X$, we obtain an equation of degree $2$ in $Y$ and $Z$; easily we have the equation
\begin{equation}
{y}^{2}-x\left ({k}^{2}-2\,k+2\right )y=x\left (x-1\right )\left (x-{k}^
{2}\right) 
\label{k}
\end{equation}
   
with  the birational transformation 
\[
Z={\frac {y}{k\left (x-1\right )}},Y=-{\frac {yk}{-y+{x}^{2}-x
}}
\]

with inverse 

\[
y={\frac {Zk\left (ZkY-Z{k}^{2}-Y\right )}{Y}},x={\frac {Zk
\left (Y-k\right )}{Y}}.
\]

The singular fibers of this fibration are
\[%
\begin{array}
[c]{cccc}%
\text{at} & k=0 & \text{of type} & I_{1}^*\\
\text{at} & k=\infty & \text{of type} & I_{12}\\
\text{at} & k=2 & \text{ of type } & I_{2}\\
\text{at} & k=4 & \text{of type} & I_{1}\\
\text{at} & k=2\,i & \text{of type} & I_{1}\\
\text{at} & k=-2\,i & \text{of type} & I_{1}%
\end{array}
\]
The rank of the Mordell-Weil group is one. The point $(x=1,y=0)$ is a non-torsion point of height $\frac{4}{3}$;  
the point $(0,0)$ is a two-torsion point and $(k,k)$ is of order $4$.

\subsubsection{Fibration of parameter $v$}
This fibration is obtained from the parameter $v=\frac{X+ZY}{Y-1}$. Eliminating $Y$ 
and using the birational transformation

\[
x={\frac {v
\left (v+X\right )\left (-v-1+Z\right )}{Z}},y=-{\frac {v\left (x-{v}^{3}-{v}^{2}\right )}{Z}}
\]
we get the equation
\begin{equation}
y^2+(v+1)^2yx-v^2(1+2v)y=(x-v)(x-v^2)(x-v^2-v^3).
\label{v}
\end{equation}


The singular fibers of this fibration are

\[%
\begin{array}
[c]{cccc}%
\text{at} & v=0 & \text{of type} & I_{8}\\
\text{at} & v=\infty & \text{of type} & I_{10}\\
\text{at} & v=v_0,  & \text{ of type } & I_{1}\\
\end{array}
\]
  
where $v_0$ denotes a root of the polynomial $t^6-5t^4+39t^2+2.$
  
The Mordell-Weil group is of rank two; the two points $(0,v^3),(v,0)$ are generators of the Mordell-Weil group  (the determinant of the heights matrix is $\frac{1}{10}$). The Mordell-Weil torsion-group is $0.$ 

\subsection{Divisors}
In this section we study the divisors of some functions. Using the elliptic fibration $(X,Y,Z)\mapsto Z=s$ 
we can compute the horizontal divisor of the following functions. 
We denote $(f)_h$ the horizontal divisor of $f$; then we have

\begin{eqnarray*}
(X)_h=-(0)-(2A)+(4A)+(6A)\\
(Y)_h=-(0)-(6A)+(4A)+(2A)\\
(X-1)_h=-(0)-(2A)+(3A)+(7A)\\
(Y-1)_h=-(0)-(6A)+(A)+(5A)\\
(X+s)_h=-(0)-(2A)+2(A)\\
(Y+s)_h=-(0)-(6A)+2(7A)\\
(X+1/s)_h=-(0)-(2A)+2(5A)\\
(Y+1/s)_h=-(0)-(6A)+2(3A)\\
(X+Y)_h=-(2A)-(6A)+2(4A)\\
(X+sY)_h=-(0)-(2A)-(6A)+(A)+(3A)+(4A)\\
\end{eqnarray*} 

\begin{proposition}
The horizontal divisors of the $7$ functions $X,Y,X-1,Y-1,Y+s,X+Y,X+sY$ generate the group of principal divisors with support in $8$-torsion sections.
\label{eq: Tor}
\end{proposition}
\begin{proof}
If we write $f_1,f_2,..,f_7$ for these functions, then  
the determinant of the matrix $(m_{i,j})$ with $m_{i,j}=ord_{iA}(f_j)$, $1  \leq  i,j \leq 7$,  is equal to $8.$ 
\end{proof}
We use the following notations: $Div(f)$ for the divisor of the function $f$ on the surface, $(f)_0$ for the divisor of the zeros of $f$ 
and $(f)_{\infty}$ for the divisor of the poles.
We get 
\[Div(Z)=\sum_{i=0}^7\Theta_{0,i}-\sum_{i=0}^7\Theta_{\infty,i}\]
\[(Z-1)_0=\sum_{i=0}^3\Theta_{1,i}\]
\[(Z+1)_0=\sum_{i=0}^1\Theta_{-1,i}.\]
 Since $X,Y,Z$ play the same role, the elliptic fibrations $(X,Y,Z)\mapsto X$ and also $(X,Y,Z)\mapsto Y$ have the same property for the singular  fibers: two singular fibers of type $I_8$ for $X=0,\infty$ and $Y=0,\infty$, one singular fiber of type $I_4$ for $X=1$, $Y=1$.  Then we can represent on the graph  the divisor of $X$, drawing two disjoint $8-$gones going throught $(0),(2A)$ and $(4A),(6A)$ and a disjoint $4-$gone throught $(3A),(7A)$. 
 We have 
 \begin{eqnarray*}
Div(X)=-(0)-\Theta_{\infty,0}-\Theta_{\infty,1}-\Theta_{\infty,2}-(2A)-\Theta_{0,6}-\Theta_{0,7}-\Theta_{0,0}\\
+(4A)+\Theta_{\infty,4}+\Theta_{\infty,5}+\Theta_{\infty,6}+(6A)+\Theta_{0,2}+\Theta_{0,3}+\Theta_{0,4}\\
\end{eqnarray*}
\begin{eqnarray*}
(X-1)_0=(3A)+\Theta_{1,3}+(7A)+\Theta_{-1,1}.
\end{eqnarray*}
A similar calculation for $Y$ gives
 \begin{eqnarray*}
Div(Y)=-(0)-\Theta_{\infty,0}-\Theta_{\infty,7}-\Theta_{\infty,6}-(6A)-\Theta_{0,2}-\Theta_{0,1}-\Theta_{0,0}\\
+(4A)+\Theta_{\infty,4}+\Theta_{\infty,3}+\Theta_{\infty,2}+(2A)+\Theta_{0,6}+\Theta_{0,5}+\Theta_{0,4}.
\end{eqnarray*}

\begin{eqnarray*}
(Y-1)_0=(A)+\Theta_{1,1}+(5A)+\Theta_{-1,1}.
\end{eqnarray*}
The fibration $(X,Y,Z)\mapsto k=X+Y$ has singular fibers of type $I_1^*,I_{12}$ at $k=0$ and $k=\infty$, so we can write the divisor of $X+Y.$
By permutation  we have also the divisors of $Y+Z$ and $X+Z$ 

\begin{eqnarray*}
Div(X+Y)=-(2A)-\Theta_{\infty,2}-\Theta_{\infty,1}-\Theta_{\infty,0}-\Theta_{\infty,7}-\Theta_{\infty,6}           \\
-(6A)-\Theta_{0,6}-\Theta_{0,7}-\Theta_{0,0}
-\Theta_{0,1}-\Theta_{0,2}\\
+\Theta_{\infty,4}+\Theta_{0,4}+2(4A)+2\Theta_{1,0}+\Theta_{1,1}+\Theta_{1,2}\\
Div(X+Z)=-(0)-\Theta_{\infty,0}-\Theta_{\infty,7}-\Theta_{\infty,6}-\Theta_{\infty,5}-\Theta_{\infty,4}\\
-\Theta_{\infty,3}-\Theta_{\infty,2}-(2A)-\Theta_{0,6}-\Theta_{0,7}-\Theta_{0,0} \\
+\Theta_{1,1}+\Theta_{-1,1}+2(A)+2\Theta_{0,3}+\Theta_{0,2}+\Theta_{0,4}\\
Div(Y+Z)=-(0)-\Theta_{\infty,0}-\Theta_{\infty,1}-\Theta_{\infty,2}-\Theta_{\infty,3}-\Theta_{\infty,4}\\
-\Theta_{\infty,5}-\Theta_{\infty,6}-(6A)-\Theta_{0,0}-\Theta_{0,1}-\Theta_{0,2}\\
+\Theta_{1,3}+\Theta_{-1,1}+2(7A)+2\Theta_{0,5}+\Theta_{0,4}+\Theta_{0,6}.
\end{eqnarray*}

At last the fibration $(X,Y,Z) \mapsto v=\frac{(X+ZY)}{(Y-1)}$ has two singular fibers of type $I_8,I_{12}$ at $v=0$ and $v=\infty$;
thus it follows 
\begin{eqnarray*}
Div(\frac{X+ZY}{Y-1})=(3A)+\Theta_{1,0}+\Theta_{1,3}+4A
+\Theta_{0,4}+\Theta_{0,3}+\Theta_{0,2}+\Theta_{0,1}\\
-(2A)-\Theta_{\infty,2}-\Theta_{\infty,1}-\Theta_{\infty,0}-\Theta_{\infty,7}
-\Theta_{\infty,6}-\Theta_{\infty,5}-(5A)-\Theta_{0,7}-\Theta_{0,6}.\\
\end{eqnarray*}

\section{Fibrations from the modular fibration}
We give a first set of elliptic fibrations with elliptic parameters belonging to the multiplicative group of functions coming from Proposition \ref{eq: Tor} plus $Z$ and $Z\pm 1$. The first ones come from some easy linear combination of divisors of functions.
The others, like $t$, come from the following remark. We can draw, on Figure  \ref{fig: G}, two disjoint subgraphs corresponding to singular fibers of the same fibration.  
We give all the details only in the case of parameter $t$. 
\subsubsection{Fibration of parameter $a$}

This  fibration is obtained with the parameter $a=\frac{Z-1}{X+Y}$. Eliminating  $Z$ in the equation and doing the birational transformation
\[
Y=-{\frac {y\left (1+a\right )}{x+xa-1}},X={\frac {x\left (x+x
a-1\right )}{y\left (1+a\right )}}
\]
with inverse

\[
y={\frac {Y\left (XY+XYa+1\right )}{1+a}},x=-XY
\]
 we get 
\begin{equation}
{y}^{2}-{\frac {\left (x-1\right )y}{\left (1+a\right )a}}
=x(x-\frac{1}{1+a})^2.
\label{a}
\end{equation}
The singular fibers of this fibration are
\[%
\begin{array}
[c]{cccc}%
\text{at} & a=0 & \text{of type} & I_{8}\\
\text{at} & a=\infty & \text{of type} & I_{1}^*\\
\text{at} & a=-1 & \text{of type} & I_{6}\\
\text{at} & a=a_0  & \text{ of type } & I_{1}\\
\end{array}
\]
  
with $a_0$  root of the ploynomial $16X^3+11X^2-2X+1.$

The point $(x=\frac{1}{1+a},y=0)$ is of height $\frac{1}{24}$.
The torsion-group of the Mordell-Weil group is $0$.
\subsubsection{Fibration of parameter $d$}
This fibration is obtained with parameter $d=X Y$ which also is equal to $-x$ in the previous Weierstrass equation.
Eliminating $X$ and making the birational transformation
\[
y=-\left (d+1\right )Y\left ({d}^{2}-x\right ),x=-ZYd\left (d+
1\right )
\]

we get

\begin{equation}
y^2-2d\,y\,x=x(x-d^2)(x-d(d+1)^2).
\label{d}
\end{equation}
The singular fibers are
\[
\begin{array}
[c]{cccc}%
\text{at} & d=0 & \text{of type} & I_{2}^*\\
\text{at} & d=\infty & \text{of type} & I_{2}^*\\
\text{at} & d=1  & \text{ of type } & I_{2}\\
\text{at} & d=-1 & \text{of type} & I_{0}^*\\
\end{array}
\]
The three points of abscisses $0,d+d^2,d^3+d^2$ are two-torsion points. The point 
$(d^2,0)$ is of height $1$.

\subsubsection{Fibration of parameter $p$}
This fibration is obtained with $p=\frac{(XY+1)Z}{X}=\frac{Vs}{U}$ which is also equal to $x/d^2$ with notation of the previous fibration.  
We start from the equation in $U,V$ and eliminating $V$ and making the birational transformation
\[
s=\frac{xp(p+1)}{y+xp},U=\frac{x}{p(p+1)}
\]
we obtain

\begin{equation}
y^2=x(x-p)(x-p(p+1)^2).
\label{p}
\end{equation}

The singular fibers are
\[
\begin{array}
[c]{cccc}%
\text{at} & p=0 & \text{of type} & I_{2}^*\\
\text{at} & p=\infty & \text{of type} & I_{4}^*\\
\text{at} & p=-2  & \text{ of type } & I_{2}\\
\text{at} & p=-1 & \text{of type} & I_{4}\\
\end{array}
\]
The Mordell-Weil group is isomorphic to $(\mathbb Z/2\mathbb Z)^2$.
\subsubsection{Fibration of parameter $w$}
Using the factorisation of the equation of $Y_2$ 
\[
\left (Z+X\right )\left (X+Y\right )\left (X-1\right )=X\left (YZ+X
\right )\left (X+Y+Z-1\right )
\]
we put $w=X+Y+Z-1=\frac{(X+Y)(X+Z)(X-1)}{X(YZ+X)}.$
Eliminating $Z$ in the equation of $Y_2$ and doing the birational transformation
\[
x=-\left (1-Y+wY\right )\left (1-X+wX\right ),y=-\left (w-1
\right )Xx
\]
we obtain the equation
\begin{equation}
y^2+w^2(x+1)y=x(x+1)(x+w^2).
\label{w}
\end{equation}
The singular fibers are
\[
\begin{array}
[c]{cccc}%
\text{at} & w=0 & \text{of type} & I_{6}\\
\text{at} & w=\infty & \text{of type} & I_{12}\\
\text{at} & w=1  & \text{ of type } & I_{2}\\
\text{at} & w=-1 & \text{of type} & I_{2}\\
\text{at} & w=\pm 2i\sqrt 2 & \text{of type} & I_{1}\\
\end{array}
\] 
The Mordell-Weil group is isomorphic to $\mathbb Z/6 \mathbb Z$, generated by $(x=-w^2,y=0).$

\subsubsection{Fibration of parameter $b$}
We put $b=\frac{XY}{Z}$. Eliminating $Z$ in the equation of $Y_2$ and using the birational transformation 

\[
x=-\frac{b(Y+b)(X+b)}{X},y=-Yb(x-(b+1)^2)
\]
we obtain the equation
\begin{equation}
y^2+2(b+1)xy+b^2(b+1)^2y=x(x+b^2)(x-(b+1)^2)
\label{b}
\end{equation}
 or with $z=x+y$ 
 
\[
z^2+2bzx+b^2(b+1)^2z=x^3.
\] 
The singular fibers are
\[
\begin{array}
[c]{cccc}%
\text{at} & b=0 & \text{of type} & IV^*\\
\text{at} & b=\infty & \text{of type} & IV^*\\
\text{at} & b=-1  & \text{ of type } & I_{6}\\
\text{at} & b=b_0 & \text{of type} & I_{1}\\
\end{array}
\]  
with $b_0$  root of the polynomial $27b^2+46b+27$.
The Mordell-Weil group is of rank $1$, the point $(x=-b^2,y=0)$ is of height $\frac{4}{3}$ and the torsion-group is of order $3$.  

\subsubsection{Fibration of parameter $r$}
Let $r=\frac{(X+Z)(Y+Z)}{ZX}$, $r$ is also equal to $x$ in the fibration of parameter $b$. Eliminating $Y$ in the equation of $Y_2$ and doing  the birational transformation.
\[
x=-{\frac {\left (rX-X-Z\right )r}{Z}},y=-{\frac {X\left (x-{r
}^{3}\right )x\left (r-1\right )}{-{r}^{2}+x}}
\] 
we obtain the equation
\begin{equation}
y^2+2(r-1)xy=x(x-1)(x-r^3).
\label{r}
\end{equation}
The singular fibers are
\[
\begin{array}
[c]{cccc}%
\text{at} & r=0 & \text{of type} & I_{2}^*\\
\text{at} & r=\infty & \text{of type} & I_{6}^*\\
\text{at} & r=1  & \text{ of type } & I_{2}\\
\text{at} & r=\pm 2i & \text{of type} & I_{1}\\
\end{array}
\]
The Mordell-Weil group is of rank $1$, the point $(1,0)$ is of height $1$. The torsion group is of order $2$, 
generated by $(0,0).$

\subsubsection{Fibration of parameter $e$}
Let $e=\frac{YX}{(Y+Z)Z}$, $e$ is also equal to $-\frac{x}{r^2}$, where $x$ is from equation (\ref{r}). 
 Eliminating $Y$ from the equation of $Y_2$ and doing the birational transformation
\[
y={\frac {-\left (2\,{e}^{2}+e+x\right )(e\left (x-2\,e-1\right )X-x\left (e+1\right ))}{\left (2\,e+1\right )\left (e+1\right )}}, x={\frac 
{-e\left (2\,e+1\right )\left (-Ze+X\right )}{X+Z}}
\] 
we obtain the equation
\begin{equation}
y^2=x(x^2-e^2(e-1)x+e^3(2e+1)).
\label{e}
\end{equation}
The singular fibers are
\[
\begin{array}
[c]{cccc}%
\text{at} & e=0 & \text{of type} & III^*\\
\text{at} & e=\infty & \text{of type} & I_{4}^*\\
\text{at} & e=-1  & \text{ of type } & I_{2}\\
\text{at} & e=-1/2 & \text{of type} & I_{2}\\
\text{at} & e=4 & \text{of type} & I_{1}\\
\end{array}
\]
The Mordell-Weil group is of rank $1$, the point $(e^3,e^3+e^4)$ is of height $1$. The torsion group is of order $2$, 
generated by $(0,0).$


\subsubsection{Fibration of parameter $f$}
Let $f={\frac {Y\left (X+Z\right )^{2}\left (Z+Y\right )}{{Z}^{3}X}}$, $f$ is also equal to $-x$ where $x$ is from (\ref{r}). 
We start from (\ref{r}) and use the transformation 
\[
y=\frac{V'}{x(x-1)},r=-\frac{U'}{x(x-1)}
\]
we obtain the equation
\begin{equation}
V'^2-2fV'U'-2f^2(f-1)V'=U'^3+f^4(f-1)^3.
\label{f}
\end{equation}
The singular fibers are
\[
\begin{array}
[c]{cccc}%
\text{at} & f=0 & \text{of type} & III^*\\
\text{at} & f=\infty & \text{of type} & II^*\\
\text{at} & f=1  & \text{ of type } & I_{4}\\
\text{at} & f=32/27 & \text{of type} & I_{1}\\
\end{array}
\]
The Mordell-Weil group is $0$.
\subsubsection{Fibration of parameter $g$}
Let $g=\frac{XY}{Z^2}$, Eliminating $Y$ in the equation of $Y_2$ and using 
the birational transformation

\[
y=-{\frac {\left (g^2-1\right )\left (-gXZ-g{Z
}^{2}-{X}^{2}+gX{Z}^{2}\right )g}{Z\left (X+Z\right )^{2}}},-x={\frac 
{g\left (g+1\right )\left (Zg+X\right )}{X+Z}}
\]
we obtain the equation
\begin{equation}
y^2=x^3+4g^2x^2+g^3(g+1)^2x.
\label{g}
\end{equation}
The singular fibers are
\[
\begin{array}
[c]{cccc}%
\text{at} & g=0 & \text{of type} & III^*\\
\text{at} & g=\infty & \text{of type} & III^*\\
\text{at} & g=-1  & \text{ of type } & I_{4}\\
\text{at} & g=1 & \text{of type} & I_{2}\\
\end{array}
\]
The Mordell-Weil group is of order $2$.
\subsubsection{Fibration of parameter $h$}
Let $h=\frac{(Y+Z)YX^2}{Z^3(X+Z)}$, we can see that $h=\frac{x}{(g+1)}$, with $x$ from (\ref{g}).
We start from (\ref{g}) and if $y=(g+1)z$ we obtain a quartic equation in $z$ and $g$ with 
 rational points $g=-1,z=\pm 2h$. Using standard transformation we obtain
 
 \[
 y^{\prime 2}+(h-\frac{1}{h}-8)x^{\prime }y^{\prime }-\frac{96}{h}y^{\prime }=\left( x^{\prime }-\frac{1}{4}(h^{2}+\frac{1}{h^{2}})+4h+\frac{8}{h}+\frac{1}{2}\right) \left( x^{\prime 2}-\frac{256}{h}\right)
 \]
or  also 

\begin{equation}
y^2=x^3-\frac{25}{3}x-h-\frac{1}{h}-\frac{196}{27}.
\label{h}
\end{equation}
The singular fibers are
\[
\begin{array}
[c]{cccc}%
\text{at} & h=0 & \text{of type} & II^*\\
\text{at} & h=\infty & \text{of type} & II^*\\
\text{at} & h=-1  & \text{ of type } & I_{2}\\
\text{at} & h=h_0 & \text{of type} & I_{1}\\
\end{array}
\]
where $h_0$ is a root of the polynomial $27h^2-446h+27$. The Mordell-Weil group is of rank $1$ without torsion.

The point $(1/16\,{h}^{2}+1/16\,{h}^{-2}+h+{h}^{-1}+{\frac {29}{24}},{\frac {1}{
64}}\,{\frac {\left (h-1\right )\left (h+1\right )\left ({h}^{4}+24\,{
h}^{3}+126\,{h}^{2}+24\,h+1\right )}{{h}^{3}}})$ is of height $4$.
We recover Elkies' result \cite{El2} cited in the introduction.

\subsubsection{Fibration of parameter $t$}

On the graph (Figure \ref{fig: G}), we can see two singular fibers of type $I_{4}^*$ of a new fibration (Figure \ref{fig: H}). They correspond to two divisors $D_1$ and $D_2$ with 
\[
D_1=\Theta_{\infty,4}+(3A)+2\Theta_{\infty,3}+2\Theta_{\infty,2}+2\Theta_{\infty,1}+2\Theta_{\infty,0}
+2(0)+\Theta_{0,0}+\Theta_{1,0}
\]
\[
D_2=(7A)+\Theta_{0,6}+2\Theta_{0,5}+2\Theta_{0,4}+2\Theta_{0,3}+2\Theta_{0,2}+2(6A)+
\Theta_{\infty,6}+\Theta_{1,2}.
\]
 We  look for  a parameter of the new fibration as a function $t$ with divisor $D_1-D_2$.


\begin{center}
\begin{tikzpicture}[scale=1]%

\draw [ thick,color=red](3.535,3.535)-- (0,5);%
\draw [thick,color=red](0,5)--(-3.535,3.535) ;%
\draw  [thick,color=red](-3.535,3.535)--(-5,0);
\draw [thick,color=red] (5,0)--(3.535,3.535); 
\draw [fill=black](0,5) circle (0.05cm)node [above] {$\Theta_{\infty,2}$};
\draw [fill=black](-3.535,3.535) circle (0.05cm)node [above] {$\Theta_{\infty,3}$};
\draw [fill=black](3.535,3.535) circle (0.05cm)node [above] {$\Theta_{\infty,1}$};
\draw [fill=black](-5,0) circle (0.05cm)node [left] {$\Theta_{\infty,4}$};
\draw [fill=black](0,-5) circle (0.05cm)node [left] {$\Theta_{\infty,6}$};
\draw [fill=black](5,0) circle (0.05cm)node [right] {$\Theta_{\infty,0}$};
\draw [ thick,dashed,color=blue](0,-3)--(2.121,2.121);
\draw [ thick,dashed,color=blue](2.121,2.121)--(-3,0);
\draw [ thick,dashed,color=blue](-3,0)--(2.121,-2.121);
\draw [ thick,dashed,color=blue](2.121,-2.121)--(0,3);
\draw [fill=black](0,3)  circle (0.05cm)node [left] {$\Theta_{0,6}$};
\draw [fill=black] (3,0) circle (0.05cm)node[below] {$\Theta_{0,0}$};
\draw [fill=black] (0,-3) circle (0.05cm)node[left] {$\Theta_{0,2}$};
\draw [fill=black]  (2.121,2.121)circle (0.05cm)node[right] {$\Theta_{0,3}$};
\draw [fill=black] (-3,0) circle (0.05cm)node[below] {$\Theta_{0,4}$};
\draw [fill=black]  (2.121,-2.121)circle (0.05cm)node[right] {$\Theta_{0,5}$};

\draw [fill=black](-2.828,2.828)circle (0.05cm)node [below] {$3A$};
\draw [fill=black](0,-4)circle (0.05cm)node [below] {$6A$};
\draw [fill=black](2.828,-2.828)circle (0.05cm)node [below] {$7A$};
\draw [fill=black](4,0) circle (0.05cm)node [below] {$0$};

\draw [thick,color=red]   (4,0)--(3,0);
\draw [ thick,dashed,color=blue]   (2.828,-2.828)--(2.121,-2.121);
\draw [ thick,dashed,color=blue]   (0,-4)--(0,-3);

\draw [ thick,color=red]   (-2.828,2.828)--(-3.535,3.535);
\draw [thick,color=red]   (4,0)--(5,0);
\draw [ thick,dashed,color=blue]   (0,-4)--(0,-5);

\draw [fill=black](0.8,0) circle (0.05cm)node [below] {$\Theta_{1,2}$};
\draw [fill=black](0,0.8) circle (0.05cm)node [below] {$\Theta_{1,0}$};


\draw [thick,color=red] (4,0)--(0,0.8);
\draw[ thick,dashed,color=blue](0,-4)--(0.8,0);
\end{tikzpicture}%
\end{center}

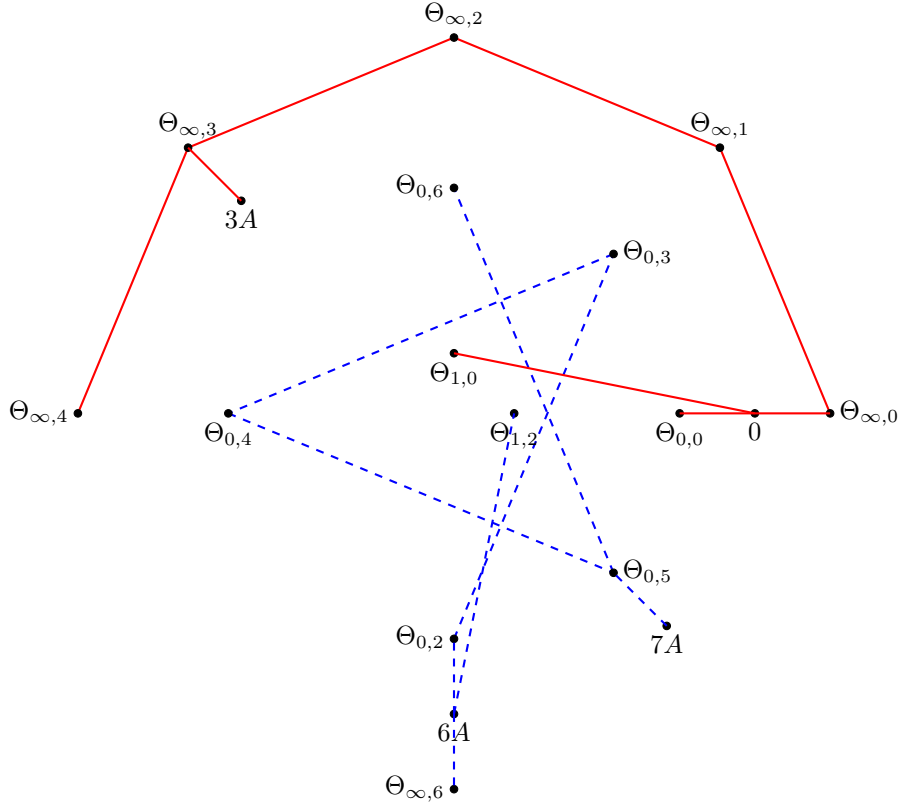
\captionof{figure}{Two singular fiber $I_4^*$}
\label{fig: H}

Let $E_{s}$ and $T_{s}$ be the generic fiber and the trivial lattice of the
fibration of parameter $s.$ For this fibration we write $D_{i}=\delta
_{i}+\Delta_{i}$ with $i=1,2$ where $\delta_{i}$ is an horizontal divisor and
$\Delta_{i}$ a vertical divisor. More precisely we have $\delta
_{1}=(3A)+2(0)$ and $\delta_{2}=(7A)+2(6A),$ and the classes of $\delta_{i}$
and $D_{i}$ are equal $\operatorname{mod}T_{s}.$ If $K=\mathbb{C}(s)$ recall
the isomorphism: $E_{s}(K)\sim NS(Y_{2})/T_{s}.$ So the class of $\delta
_{1}-\delta_{2}$ is $0$ in $NS(Y_{2})/T_{s};$ thus there is a function
$t_{0}=\frac{X^{2}\left(  Y+Z\right)  }{(X-1)(X+Y)}$ with divisor  $\delta
_{1}-\delta_{2}.$ We can choose $t$ as $t=t_{0}Z^{a}(Z-1)^{b}$. The rational integers $a$ and $b$
can be computed using the divisors of section $7.4.$ We find $a=0,b=1$ and
we can take $t=\frac{X^{2}\left(  Y+Z\right)  (Z-1)}{(X-1)(X+Y)}.$
We have also 
\[
t=-X{\frac {Y\left ({Z}^{2}-Z+YZ+1\right )Z}{\left (Z-1\right )\left (YZ
+1\right )}}-{\frac {{Z}^{2}\left (Y+1\right )}{\left (Z-1\right )
\left (YZ+1\right )}}.
\]
Eliminating $X$ in the equation of $Y_2$ and then doing the birational transformation
\[
Z=\frac{W}{WT+1},\,Y=\frac{-(WT^2+T+1)}{WT+1}
\]
of inverse 
\[
T=\frac{Y+1}{Z-1},\,W=\frac{-Z(Z-1)}{YZ+1}
\]
we obtain an equation of degree $2$ in $T$. After some classical transformation we get a quartic
 with a rational point corresponding to $(T=-1+t,W=0)$ and then using a standard transformation we obtain 
\begin{equation}
y^2=x^3+t(t^2+1+4t)x^2+t^4x.
\label{t}
\end{equation}
The singular fibers are
\[
\begin{array}
[c]{cccc}%
\text{at} & t=0 & \text{of type} & I_{4}^*\\
\text{at} & t=\infty & \text{of type} & I_{4}^*\\
\text{at} & t=-1  & \text{ of type } & I_{2}\\
\text{at} & t=t_0 & \text{of type} & I_{1}\\
\end{array}
\]
where $t_0$ is a root of the polynomial $Z^2+6Z+1.$

The Mordell-Weil group is of rank one and the point $(-t^3,2t^4)$ is of height $1$. The torsion group is of order $2.$
\subsubsection{Fibration of parameter $l$}
Let $l=\frac{Z(YZ+X)}{X(1+YZ)}.$ Eliminating $Y$ in the equation of $Y_2$ and using the variable $W=\frac{Z+X}{X-1}$, we have an equation of bidegree $2$ in $W$ and $Z$; easily we obtain

\begin{equation}
y^2-yx-2l^3y=(x+l^3)(x+l^2)(x-l+l^3).
\label{l}
\end{equation}

The singular fibers are
\[
\begin{array}
[c]{cccc}%
\text{at} & l=0 & \text{of type} & I_{10}\\
\text{at} & l=\infty & \text{of type} & I_{3}^*\\
\text{at} & l=l_0  & \text{ of type } & I_{1}\\
\end{array}
\] 
with $l_0$ root of the polynomial $16\,{x}^{5}-32\,{x}^{4}-24\,{x}^{3}-23\,{x}^{2}+12\,x-2.$
The Mordell-Weil group is of rank $2$, without torsion; the two points $(-l^3,0)$ and $(-l^2,0)$ are independent and the determinant of the matrix of heights is equal to $\frac{1}{5}$. 

\section{ A second set of fibrations: gluing and breaking}
\subsection{Classical examples}
In the next section we give fibrations obtained using Elkies idea given in
\cite{El} page 11 and explained in \cite{Kum} Appendix A. If we have two
fibrations with fiber $F$ and $F^{\prime}$ satisfying $F \cdot F^{\prime}=2$
the authors explain how we get a parameter from a Weierstrass equation
of one fibration. Decomposing $F'$ into vertical and horizontal component, $F'=F'_h+F'_v$ they use $F'_h$ to construct a function on the generic fiber.  

\subsubsection{Fibration of parameter $o$} 
Starting with a fibration with two singular fibers of type $II^*$ and the $(0)$ section we obtain a fibration with a singular fiber of type $I_{12}^*$. Starting from (\ref{h}) we take $x$ as new parameter. For simplicity, let $o=x+\frac{5}{3}$, we get
\begin{equation}
y^2=x^3+(o^3-5o^2+2)x^2+x.
\label{o}
\end{equation}

The singular fibers are
\[
\begin{array}
[c]{cccc}%
\text{at} & o=0 & \text{of type} & I_{2}\\
\text{at} & o=\infty & \text{of type} & I_{12}^*\\
\text{at} & o=1  & \text{ of type } & I_{1}\\
\text{at} & o=5  & \text{ of type } & I_{1}\\
\text{at} & o=o_0 & \text{of type} & I_{1}\\
\end{array}
\] 
where $o_0$ is a  root of $x^2-4x-4$.
The Mordell-Weil group is of rank $1$, the torsion group is of order $2.$ 

The point 
 $(1/16\,\left (o-4\right )^{2}\left (o-2\right )^{2},{\frac {1}{64}}\,
\left (o-4\right )\left (o-2\right )\left ({o}^{4}-4\,{o}^{3}-20\,{o}^
{2}+96\,o-80\right ))$ is of height $4.$ 

\subsubsection{Fibration of parameter $q$}
We start with a fibration with  singular fibers of type $II^*$ and $III^*$, join them with the zero-section and obtain a singular fiber of type $I_{10}^*$ of a new fibration.
We transform the equation (\ref{f}) to obtain
\[
y'^2=x'^3+(\frac{5}{3}-\frac{2}{f})x'+f+\frac{5}{3f}-\frac{70}{27}.
\]
We take $x'$ as the parameter of the new fibration; more precisely, for simplicity, let  $q=x'-\frac{1}{3}$. We obtain 
\begin{equation}
y^2=x^3+(q^3+q^2+2q-2)x^2+(1-2q)x.
\label{q}
\end{equation}
The singular fibers are
\[
\begin{array}
[c]{cccc}%
\text{at} & q=0 & \text{of type} & I_{4}\\
\text{at} & q=\infty & \text{of type} & I_{10}^*\\
\text{at} & q=\frac{1}{2}   & \text{ of type } & I_{2}\\
\text{at} & q=q_0   & \text{ of type } & I_{1}\\
\end{array}
\]
where $q_0$ is a root of $X^2+2X+5$.
The Mordell-Weil group is of order $2.$

\subsection{Fibration with singular fiber of type $I_n$ with $n$ great}
We start with a fibration with two fibers  $I_n^*$ and $I_m^*$ and a two-torsion section. Gluing them, we can construct a fibration with a singular fiber of type $I_{n+m+8}$. The parameter will be $\frac{y}{x}$ in a good model of the first fibration. We can also start from a fibration with two singular fibers of type $I_n^*$ and $I_2$ and a two-torsion section, join them with the zero-section and obtain a new fibration with a singular fiber of type $I_{n+6}$ or $I_{n+7}$.    
\subsubsection{Fibration of parameter $m$}
We start from the fibration with parameter $t$. With the two-torsion section and the two singular fibers of type $I_4^*$, we  can form a singular fiber of of type $I_{16}$ of a new fibration (Figure \ref{fig: M}). 
\bigskip
\begin{center}
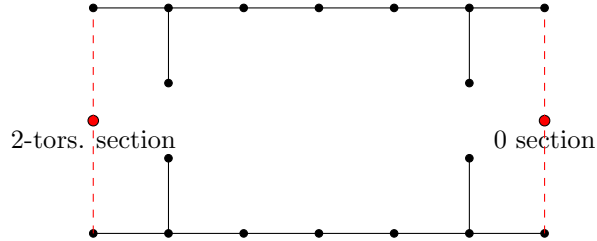

\begin{tikzpicture}[scale=1]%
\draw (-3,-2)--(3,-2);
\draw (-2,-1)--(-2,-2);
\draw(2,-1)--(2,-2);
\draw [fill=black](-3,-2) circle (0.05cm);
\draw [fill=black](-2,-2) circle (0.05cm);
\draw [fill=black](-1,-2) circle (0.05cm);
\draw [fill=black](0,-2) circle (0.05cm);
\draw [fill=black](1,-2) circle (0.05cm);
\draw [fill=black](2,-2) circle (0.05cm);
\draw [fill=black](3,-2) circle (0.05cm);
\draw [fill=black](-2,-1) circle (0.05cm);
\draw [fill=black](2,-1) circle (0.05cm);
\draw [fill=red](-3,-0.5) circle (0.07cm)node [below]{2-tors. section};
\draw [fill=red](3,-0.5) circle (0.07cm)node [below]{0 section};

\draw [dashed,red] (-3,-2)--(-3,1);
\draw [dashed,red](3,-2)--(3,1);

\draw(-3,1)--(3,1);
\draw (-2,0)--(-2,1);
\draw(2,0)--(2,1);
\draw [fill=black](-3,1) circle (0.05cm);
\draw [fill=black](-2,1) circle (0.05cm);
\draw [fill=black](-1,1) circle (0.05cm);
\draw [fill=black](0,1) circle (0.05cm);
\draw [fill=black](1,1) circle (0.05cm);
\draw [fill=black](2,1) circle (0.05cm);
\draw [fill=black](3,1) circle (0.05cm);
\draw [fill=black](-2,0) circle (0.05cm);
\draw [fill=black](2,0) circle (0.05cm);
\end{tikzpicture}%
\end{center}
\captionof{figure}{$I_4^*,I_4^* \rightarrow I_{16}$}
\label{fig: M}


From the equation (\ref{t}) we get 

\[
y'^2=x'^3+(t+\frac{1}{t}+4)x'^2+x'.
\]   
Let $m=\frac{y'}{x'}$, we obtain 

\begin{equation}
y^2+(m-2)(m+2)yx=x(x-1)^2.
\label{m}
\end{equation}
 
The singular fibers are
\[
\begin{array}
[c]{cccc}%
\text{at} & m=0 & \text{of type} & I_{2}\\
\text{at} & m=\infty & \text{of type} & I_{16}\\
\text{at} & m=\pm 2  & \text{ of type } & I_{2}\\
\text{at} & m=\pm 2\sqrt 2  & \text{ of type } & I_{1}\\
\end{array}
\] 
The Mordell-Weil group is cyclic of order $4$ generated by $(x=1,y=0)$.
\subsubsection{Fibration of parameter $n$}
With a similar method, we can start from a fibration with two singular fibers of type $I_2^*$ and $I_6^*$,  a two-torsion section and join them to have a fiber of type $I_{16}$.

\bigskip
\begin{center}
\begin{tikzpicture}[scale=1]%
\draw (-4,-2)--(4,-2);
\draw (-3,-1)--(-3,-2);
\draw(3,-1)--(3,-2);
\draw [fill=black](-4,-2) circle (0.05cm);
\draw [fill=black](4,-2) circle (0.05cm);

\draw [fill=black](-3,-2) circle (0.05cm);
\draw [fill=black](-2,-2) circle (0.05cm);
\draw [fill=black](-1,-2) circle (0.05cm);
\draw [fill=black](0,-2) circle (0.05cm);
\draw [fill=black](1,-2) circle (0.05cm);
\draw [fill=black](2,-2) circle (0.05cm);
\draw [fill=black](3,-2) circle (0.05cm);
\draw [fill=black](-3,-1) circle (0.05cm);
\draw [fill=black](3,-1) circle (0.05cm);
\draw [fill=red](-3.0,-0.5) circle (0.07cm)node [below]{2-tors. section};
\draw [fill=red](3.0,-0.5) circle (0.07cm)node [below]{0 section};

\draw [dashed,red] (-4,-2)--(-2,1);
\draw [dashed,red](4,-2)--(2,1);

\draw(-2,1)--(2,1);
\draw (-1,0)--(-1,1);
\draw(1,0)--(1,1);
\draw [fill=black](-2,1) circle (0.05cm);
\draw [fill=black](-1,1) circle (0.05cm);
\draw [fill=black](0,1) circle (0.05cm);
\draw [fill=black](1,1) circle (0.05cm);
\draw [fill=black](2,1) circle (0.05cm);
\draw [fill=black](-1,0) circle (0.05cm);
\draw [fill=black](1,0) circle (0.05cm);
\end{tikzpicture}%
\end{center}

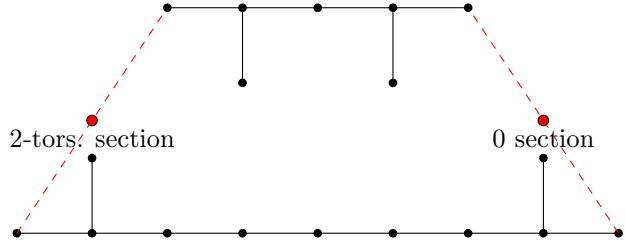
\captionof{figure}{$I_6^*,I_2^* \rightarrow I_{16}$}
\label{fig: U}

From the equation (\ref{r}) we obtain
\[
y'^2=x'^3+(r-1+\frac{2}{r})x'^2-\frac{x'}{r}.
\]
Let $n=\frac{y'}{x'}$. The Weierstrass  equation is 
\begin{equation}
y^2+(n^2-1)yx-y=x^3-2x^2.
\label{n}
\end{equation}
The singular fibers are
\[
\begin{array}
[c]{cccc}%
\text{at} & n=0 & \text{of type} & I_{2}\\
\text{at} & n=\infty & \text{of type} & I_{16}\\
\text{at} & n=n_0   & \text{ of type } & I_{1}\\
\end{array}
\]
where $n_0$ is a root of the polynomial $2x^6-9x^4-17x^2+125$.
The Mordell-Weil group is of rank $2$. The determinant of the matrix of heights of the two points $(1 \pm t,0)$ is  $\frac{3}{8}$.

\subsubsection{Fibration of parameter $j$}
Instead of the two-torsion section we can use the section of infinite order $(-t,-2t)$ in the fibration of parameter $t$. 
\bigskip
\begin{center}

\begin{tikzpicture}[scale=1]%
\draw [dotted,red](-2,-2)--(-2,2);
\draw  [dotted,red](-2,2)--(2,2);
\draw [dotted,red](2,2)--(2,-2);
\draw [dotted,red](-2,-2)--(2,-2);

\draw [ thick](-2,-2)-- (-2,1);%
\draw [thick](-1,2)--(2,2) ;%
\draw  [thick](2,2)--(2,-1);
\draw [thick](-2,-2)--(1,-2);

\draw [thick](0,2)--(0,1);
\draw [thick](0,-2)--(0,-1);
\draw [thick](2,0)--(1,0);
\draw [thick](-2,0)--(-1,0);

\draw [fill=black](-2,-2) circle (0.05cm);

\draw [fill=black](-1,-2) circle (0.05cm);

\draw [fill=black](0,-2) circle (0.05cm);

\draw [fill=black](1,-2) circle (0.05cm);

\draw [fill=red](2,-2) circle (0.07cm) ;

\draw [fill=black](-2,-1) circle (0.05cm);

\draw [fill=black](-2,0) circle (0.05cm);

\draw [fill=black](-2,1) circle (0.05cm);

\draw [fill=black](2,-2) circle (0.05cm);

\draw [fill=black](2,-1) circle (0.05cm);

\draw [fill=black](2,0) circle (0.05cm);

\draw [fill=black](2,1) circle (0.05cm);

\draw [fill=black](2,2) circle (0.05cm);

\draw [fill=black](-2,2) circle (0.07cm);

\draw [fill=black](-1,2) circle (0.05cm);

\draw [fill=black](0,2) circle (0.05cm);

\draw [fill=black](1,2) circle (0.05cm);

\draw [fill=black](2,2) circle (0.05cm);

\draw [fill=black](0,1) circle (0.05cm);

\draw [fill=black](0,-1) circle (0.05cm);

\draw [fill=black](1,0) circle (0.05cm);

\draw [fill=black](-1,0) circle (0.05cm);
\draw [fill=black](0,0) circle (0.07cm);
\draw [dotted,red](-1,0)--(1,0);

\draw [dashed] (-2.5,-2.5)--(-2.5,0.5);
\draw [dashed] (-2.5,-2.5)--(2.5,-2.5);
\draw [dashed] (2.5,-2.5)--(2.5,0.5);
\draw [dashed] (-2.5,0.5)--(2.5,0.5);

\end{tikzpicture}
\end{center}

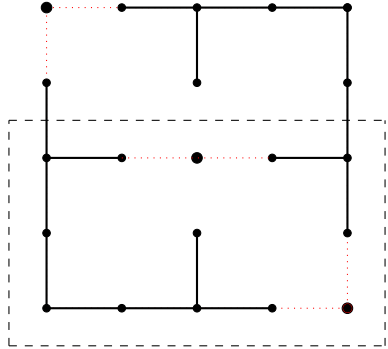
\captionof{figure}{$2I_4^* \rightarrow I_{12}$}
\label{fig: P}

In (\ref{t}), let $U'=x+t,V'=y+2t$ and take the parameter $\frac{V'}{U'}$. For simplification let $j=\frac{V'}{U'}-2=\frac{y-2x}{x+t}$; the new fibration obtained has a $3$-torsion point which can be put in $(0,0)$. So it follows

\begin{equation}
{y}^{2}-(j^2+4j)x\,y+{j}^{2}y=x^3. 
\label{j} 
\end{equation}

The singular fibers are
\[
\begin{array}
[c]{cccc}%
\text{at} & j=0 & \text{of type} & IV^*\\
\text{at} & j=\infty & \text{of type} & I_{12}\\
\text{at} & j=-1 & \text{of type} & I_{2}\\
\text{at} & j=j_0 & \text{of type} & I_{1}\\
\end{array}
\]
where $j_0$ is a root of the polynomial $(x^2+10x+27)$.
The Mordell-Weil group is isomorphic to $\mathbb{Z}/3\mathbb{Z}.$

\subsubsection{Fibration of parameter $c$}
We start from the fibration of parameter $o$ with the equation (\ref{o}). For $o=0$ we have a singular fiber of type $I_2$, the singular point of the bad reduction is $(x=1,y=0)$ so we put $x=1+u$ and  obtain the equation
\[
y^2=u^3+(-1-5o^2+o^3)u^2+2o^2(o-5)u-o^2(o-5).
\]  
The two-torsion section cut the singular fiber $I_2$ on the zero component, as shown on Figure \ref{fig: Q}.
\bigskip
\begin{center}
\begin{tikzpicture}[scale=0.8]%
\draw [thick] (-7,0)--(7,0);
\draw [thick] (-6,0)--(-6,1);
\draw [thick]  (6,0)--(6,1);
\draw [fill=black](-7,0) circle (0.05cm);
\draw [fill=black](-6,0) circle (0.05cm);
\draw [fill=black](-5,0) circle (0.05cm);
\draw [fill=black](-4,0) circle (0.05cm);
\draw [fill=black](-3,0) circle (0.05cm);
\draw [fill=black](-2,0) circle (0.05cm);
\draw [fill=black](-1,0) circle (0.05cm);
\draw [fill=black](0,0) circle (0.05cm);
\draw [fill=black](1,0) circle (0.05cm);
\draw [fill=black](2,0) circle (0.05cm);
\draw [fill=black](3,0) circle (0.05cm);
\draw [fill=black](4,0) circle (0.05cm);
\draw [fill=black](5,0) circle (0.05cm);
\draw [fill=black](6,0) circle (0.05cm);
\draw [fill=black](7,0) circle (0.05cm);
\draw [fill=black](6,1) circle (0.05cm);
\draw [fill=black](-6,1) circle (0.05cm);

\draw [fill=black](0,3) circle (0.07cm)node[right]{0-section};
\draw [fill=black](0,4) circle (0.05cm);
\draw [fill=black](-2,4) circle (0.05cm);
\draw [ thick] (-2,4)--(0,4);
\draw[thick,dotted,red] (0,3)--(0,4);
\draw[thick,dotted,red] (6,1)--(0,3);
\draw [fill=black](-2,3) circle (0.07cm)node[left]{2-tors. section};
\draw[thick,dotted,red] (-2,3)--(0,4);
\draw [thick,dotted,red](-2,3)--(-6,1);


\end{tikzpicture}%
\end{center}

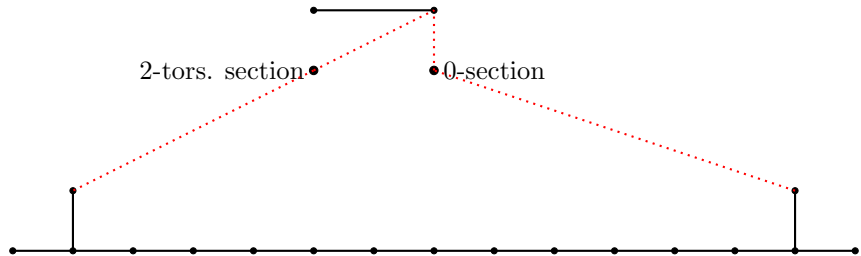
\captionof{figure}{$I_{12}^*,I_2 \rightarrow I_{18}$}
\label{fig: Q}

Let $c=\frac{y}{ox}$ with $y,x$ from the equation (\ref{o}), we have easily the Weierstrass equation of the fibration. After some translation we can suppose the  three-torsion point is $(0,0)$; we get 

\begin{equation}
y^2+(c^2+5)yx+y=x^3.
\label{c}
\end{equation}

The singular fibers are
\[
\begin{array}
[c]{cccc}%
\text{at} & c=\infty & \text{of type} & I_{18}\\
\text{at} & c=c_0 & \text{of type} & I_{1}\\
\end{array}
\]
where $c_0$ is a root of the polynomial $(x^2+2)(x^2+x+7)(x^2-x+7).$ 
The rank of the Mordell-Weil group is one. The height of $(x=\frac{-1}{4}(c^4+c^2+1),y=\frac{1}{8}(c^2-c+1)^3)$ is equal to $4.$ 

\subsection{Fibrations with singular fiber of type $I_n^*$}

In this paragraph we obtain new fibrations by gluing two fibers of type $I_{p}^{*}$ and $I_{q}^{*}$ to obtain, with the zero section, a singular fiber of type $I_{p+q+4}^{*}$ or $I_{p+4}^{*}.$ 
\subsubsection{Fibration of parameter $u$}
We start from the fibration with parameter $t$. With  the two singular fibers of type $I_4^*$, and the $0$-section, we  can form a singular fiber of  type $I_{8}^*$ of a new fibration (Figure \ref{fig: R}). 
\bigskip
\begin{center}

\begin{tikzpicture}[scale=1]%
\draw [dotted,red](-2,-2)--(-2,2);
\draw  [dotted,red](-2,2)--(2,2);
\draw [dotted,red](2,2)--(2,-2);
\draw [dotted,red](-2,-2)--(2,-2);

\draw [ thick](-2,-2)-- (-2,1);%
\draw [thick](-1,2)--(2,2) ;%
\draw  [thick](2,2)--(2,-1);
\draw [thick](-2,-2)--(1,-2);

\draw [thick](0,2)--(0,1);
\draw [thick](0,-2)--(0,-1);
\draw [thick](2,0)--(1,0);
\draw [thick](-2,0)--(-1,0);

\draw [fill=black](-2,-2) circle (0.05cm);

\draw [fill=black](-1,-2) circle (0.05cm);

\draw [fill=black](0,-2) circle (0.05cm);

\draw [fill=black](1,-2) circle (0.05cm);

\draw [fill=red](2,-2) circle (0.07cm) ;

\draw [fill=black](-2,-1) circle (0.05cm);

\draw [fill=black](-2,0) circle (0.05cm);

\draw [fill=black](-2,1) circle (0.05cm);


\draw [fill=black](2,-2) circle (0.05cm);

\draw [fill=black](2,-1) circle (0.05cm);

\draw [fill=black](2,0) circle (0.05cm);

\draw [fill=black](2,1) circle (0.05cm);

\draw [fill=black](2,2) circle (0.05cm);

\draw [fill=black](-2,2) circle (0.07cm);

\draw [fill=black](-1,2) circle (0.05cm);

\draw [fill=black](0,2) circle (0.05cm);

\draw [fill=black](1,2) circle (0.05cm);

\draw [fill=black](2,2) circle (0.05cm);

\draw [fill=black](0,1) circle (0.05cm);

\draw [fill=black](0,-1) circle (0.05cm);

\draw [fill=black](1,0) circle (0.05cm);

\draw [fill=black](-1,0) circle (0.05cm);

\draw [dashed] (-2.5,-2.5)--(-2.5,1.5);
\draw [dashed] (-2.5,-2.5)--(2.5,-2.5);
\draw [dashed] (2.5,-2.5)--(2.5,1.5);
\draw [dashed] (-2.5,1.5)--(-0.5,1.5);
\draw [dashed] (-0.5,1.5)--(-0.5,-1.5);
\draw [dashed] (-0.5,-1.5)--(0.5,-1.5);
\draw [dashed] (0.5,-1.5)--(0.5,1.5);
\draw [dashed] (0.5,1.5)--(2.5,1.5);

\end{tikzpicture}
\end{center}

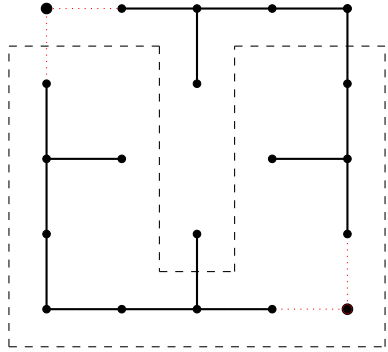
\captionof{figure}{$2I_4^* \rightarrow I_{8}^*$}
\label{fig: R}

From the equation (\ref{t}), it follows 
\[
\frac{y^2}{x^2t^2}=\frac{x}{t^2}+t+\frac{t^2}{x}+\frac{1}{t}+4
\]
Taking $u=\frac{x}{t^2}+t$ as the new parameter, we obtain

\begin{equation}
y'^2=x'^3+u(u^2+4u+2)x'^2+u^2x'.
\label{u}
\end{equation}
The singular fibers are
\[
\begin{array}
[c]{cccc}%
\text{at} & u=0 & \text{of type} & I_{1}^*\\
\text{at} & u=\infty & \text{of type} & I_{8}^*\\
\text{at} & u=-2   & \text{ of type } & I_{2}\\
\text{at} & u=-4   & \text{ of type } & I_{1}\\
\end{array}
\] 
The Mordell-Weil group is of order $2$.
\subsubsection{Fibration of parameter $i$}
We start from the fibration of parameter $s$ and from equation (\ref{s}).
With  the two singular fibers of type $I_8^*$ and $I_1^*$, and the $0$ section, we  can form a singular fiber of  type $I_{13}^*$ of a new fibration.
 We seek for a parameter of the form $\frac{x}{s^2}+\frac{k}{s}$, with $k$ chosen to have a quartic equation. We see that  $k=1$ is a good choice so  the new parameter is $i=\frac{x}{s^2}+\frac{k}{s}$ and a Weierstrass equation is

\begin{equation}
y^2={x}^{3}+\left ({i}^{3}+4\,{i}^{2}+2\,i\right ){x}^{2}+\left (-2\,{i}^{
2}-8\,i-2\right )x+i+4.
\label{i}
\end{equation}
The singular fibers are
\[
\begin{array}
[c]{cccc}%
\text{at} & i=\infty & \text{of type} & I_{13}^*\\
\text{at} & i=-\frac{5}{2}   & \text{ of type } & I_{2}\\
\text{at} & i=i_0   & \text{ of type } & I_{1}\\
\end{array}
\] 
where $i_0$ is a root of the polynomial $4x^3+11x^2-8x+16$.
The Mordell-Weil group is $0.$

\subsection{Breaking}
In this paragraph we  give a fibration obtained by breaking a singular fiber $I_{18}$ and using the three-torsion points as for the fibration of parameter $t.$  
\subsubsection{Fibration of parameter $\psi$}


\bigskip

\begin{center}
\begin{tikzpicture}[scale=1.1]%

\draw [thick](-3,0)--(3,0);
\draw [thick] (-3,0)--(0,2);
\draw [thick] (0,2)--(3,0);

\draw [fill=black](-3,0) circle (0.05cm);
\draw [fill=black](-2,0) circle (0.05cm);
\draw [fill=black](-1,0) circle (0.05cm);
\draw [fill=black](0,0) circle (0.05cm);
\draw [fill=black](1,0) circle (0.05cm);
\draw [fill=black](2,0) circle (0.05cm);
\draw [fill=black](3,0) circle (0.05cm);
\draw [fill=black](-3,0) circle (0.05cm);
\draw [fill=black](-2.5,0.333) circle (0.05cm);
\draw [fill=black](-2,0.666) circle (0.05cm);
\draw [fill=black](-1.5,1) circle (0.05cm);
\draw [fill=black](-1,1.332) circle (0.05cm);
\draw [fill=black](-0.5,1.65) circle (0.05cm);
\draw [fill=black](0,2) circle (0.05cm);

\draw [fill=black](3,0) circle (0.05cm);
\draw [fill=black](2.5,0.333) circle (0.05cm);
\draw [fill=black](2,0.666) circle (0.05cm);
\draw [fill=black](1.5,1) circle (0.05cm);
\draw [fill=black](1,1.332) circle (0.05cm);
\draw [fill=black](0.5,1.65) circle (0.05cm);

\draw [fill=black](-4,0) circle (0.07cm)node[left]{$(1)$};
\draw [fill=black](4,0) circle (0.07cm)node [right]{$(1)$};
\draw [fill=black](0,3) circle (0.07cm)node [right]{0-section}node[left]{$(-2)$};
\draw[dashed,red](-4,0)--(-3,0);
\draw[dashed,red](3,0)--(4,0);
\draw[dashed,red](0,2)--(0,3);

\draw[dotted] (-2.5,0.83)--(2.5,0.83);
\draw[dotted] (-2.5,0.83)--(-2.5,3.5);
\draw[dotted] (2.5,0.83)--(2.5,3.5);
\draw[dotted] (-2.5,3.4)--(2.5,3.4);
\draw[dotted] (-4.5,-0.5)--(4.5,-0.5);
\draw[dotted] (-4.5,-0.5)--(-4.5,0.5);
\draw[dotted] (4.5,-0.5)--(4.5,0.5);
\draw[dotted] (-4.5,0.5)--(4.5,0.5);
\draw (0,3.5) circle(0.0007cm);
\end{tikzpicture}
\end{center}

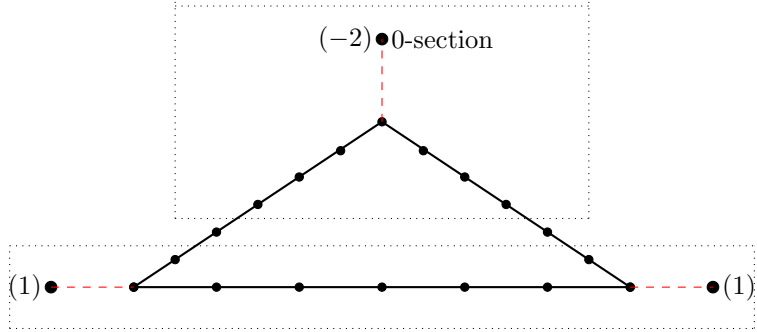
\captionof{figure}{A singular fiber $I_{18}\rightarrow III^*,I_{6}^*$}
\label{fig: K}

We start with the fibration (\ref{c}) of parameter $c$ and $3$-torsion sections . We represent the graph of the singular fiber $I_{18}$, the zero section  and two $3$-torsion sections (Figure \ref{fig: K}). On this graph we can draw two singular fibers $III^*$ and $I_{6}^*$.  The function $x$ of  (\ref{c}) has the horizontal divisor $-2(0)+P+(-P)$ if $P$ denotes the $3$-torsion point and we can take it as the parameter $\psi$ of the new fibration. We get the equation 

\begin{equation}
y^2={x}^{3}-5\,{x}^{2}{\psi}^{2}-\psi\,{x}^{2}-{\psi}^{5}x.
\label{psi}
\end{equation}
The singular fibers are
\[
\begin{array}
[c]{cccc}%
\text{at} & \psi=\infty & \text{of type} & I_{6}^*\\
\text{at} & \psi=-\frac{5}{2}   & \text{ of type } & III^*\\
\text{at} & \psi=-\frac{1}{4}   & \text{ of type } & I_{1}\\
\text{at} & \psi=\psi_0   & \text{ of type } & I_{1}\\
\end{array}
\] 
with $\psi_0$   root of the polynomial $x^2+6x+1$.
The Mordell-Weil group is of rank $1$. The height of the point 
$$(1/4\,\left ({\psi}^{2}+3\,\psi+1\right )^{2},-1/8\,\left ({\psi}^{2}+
3\,\psi+1\right )\left ({\psi}^{4}+6\,{\psi}^{3}+{\psi}^{2}-4\,\psi-1
\right ))$$ is $4.$ The torsion-group is of order $2.$

\section{Last set}
From the first set of fibrations we see that not all the components of singular fibers defined on $\mathbb {Q}$ appear 
on the graph of the Figure \ref{fig: G}. We have to introduce some of them to construct easily the last fibrations.
For example, we start with the fibration of parameter $p$ (\ref{p}). Using the Figure \ref{fig: G} we see only $3$ components of the singular fiber $I_4$ for $p=-1$ i.e. $\Theta_{0,1},\Theta_{0,2},\Theta_{0,3}$. The fourth is the rational curve named  $\Theta_{p,-1,3}$ parametrized by
\[
X=4\,{\frac {\left (w+1\right )w}{1+3\,{w}^{2}}},\, Y=1/4\,{\frac {1+3\,{w}^{2}}{\left (-1+w\right )w}},\, Z=-2\,{\frac {\left (-1+w\right )\left (w+1\right )}{1+3\,{w}^{2}}}. 
\]
\bigskip
\begin{center}
\begin{tikzpicture}[scale=1.5]
\draw[thick](-3,-2)--(3,-2);
\draw[thick](-2,-2)--(-2,-1);
\draw[thick](2,-2)--(2,-1);

\draw [fill=black](-3,-2) circle (0.05cm)node[below]{$4A$};
\draw [fill=black](-2,-2) circle (0.05cm)node[below]{$\Theta_{\infty,4}$};
\draw [fill=black](-1,-2) circle (0.05cm)node[below]{$\Theta_{\infty,5}$};
\draw [fill=black](0,-2) circle (0.05cm)node[below]{$\Theta_{\infty,6}$};
\draw (-0.1,-1.8)node [right] {$p=\infty$};
\draw [fill=black](1,-2) circle (0.05cm)node[below]{$\Theta_{\infty,7}$};
\draw [fill=black](2,-2) circle (0.05cm)node[below]{$\Theta_{\infty,0}$};
\draw [fill=black](3,-2) circle (0.05cm)node[below]{$\Theta_{\infty,1}$};
\draw [fill=black](-2,-1) circle (0.05cm)node[right]{$\Theta_{\infty,3}$};
\draw [fill=black](2,-1) circle (0.05cm)node[right]{$0$};

\draw[thick](-2,4)--(2,4);
\draw[thick](-1,4)--(-1,3);
\draw[thick](1,4)--(1,3);

\draw [fill=black](-2,4) circle (0.05cm)node[above]{$\Theta_{0,5}$};
\draw [fill=black](-1,4) circle (0.05cm)node[above]{$\Theta_{0,6}$};
\draw [fill=black](0,4) circle (0.05cm)node[above]{$2A$};
\draw (-0.1,3.8)node [right]{$p=0$}; 
\draw [fill=black](1,4) circle (0.05cm)node[above]{$\Theta_{1,2}$};
\draw [fill=black](2,4) circle (0.05cm)node[above]{$\Theta_{1,1}$};
\draw [fill=black](-1,3) circle (0.05cm)node[above]{$\Theta_{0,7}$};
\draw [fill=black](1,3) circle (0.05cm)node[above]{$\Theta_{1,3}$};

\draw [fill=black](-2.65,0) circle (0.07cm)node[above]{$\Theta_{0,4}$};
\draw [fill=black](2.65,0) circle (0.07cm)node[right]{$A$};
\draw [fill=black](-1.24,0) circle (0.07cm)node[above]{$3A$};
\draw [fill=black](1.24,0) circle (0.07cm)node[above]{$\Theta_{0,0}$};

\draw [thin,dashed,red](-3,-2)--(-2,4);
\draw [thin,dashed,red](3,-2)--(2,4);
\draw [thin,dashed,red](-2,-1)--(1,3);
\draw [thin,dashed,red](2,-1)--(-1,3);

\draw [thick] (-1.5,1.4)--(1.5,1.4);
\draw [thick] (-1.5,1.4)--(0,2.4);
\draw [thick] (0,2.4)--(1.5,1.4);
\draw [fill=black](-1.5,1.4) circle (0.05cm)node[left]{$\Theta_{0,1}$};
\draw [fill=black](0,1.4) circle (0.05cm)node[above]{$\Theta_{0,2}$};
\draw [fill=black](1.5,1.4) circle (0.05cm)node[right]{$\Theta_{0,3}\,\,p=-1$};
\draw [fill=black](0,2.4) circle (0.05cm)node[above]{$\Theta_{p,-1,3}$};

\draw [thin,dashed,red](-1.5,1.4)--(-1.24,0);
\draw [thin,dashed,red](-1.5,1.4)--(1.24,0);
\draw [thin,dashed,red] (1.5,1.4)--(-2.65,0);
\draw [thin,dashed,red] (1.5,1.4)--(2.65,0);
\end{tikzpicture}
\end{center}
\captionof{figure}{Fibration of parameter $p$: three singular fibers}
\label{fig: K1}%

The four components  $\Theta_{0,0},\Theta_{0,4},A,3A$ are sections of the fibration of parameter $p$ which cut the singular fibers following the previous figure.

\subsection{From fibration of parameter $p$}
\subsubsection{Fibration of parameter $\delta$}

On the  Figure \ref{fig: K1} 
 we can see the divisor $\Delta$,  
\[
\Delta =6\Theta_{\infty,0}+5\Theta_{\infty,7}+4\Theta_{\infty,6}+3\Theta_{\infty,5}+2\Theta_{\infty,4}+\Theta_{\infty,3}
+3\Theta_{\infty,1}+40+2\Theta_{0,0}.
\]

The divisor $\Delta$  corresponds to a fiber of type $II^*.$
Using the equation (\ref{p}) and the previous remark, we can calculate the divisors of $p,p+1$ and $Up(p+1)$. The poles  of  $\delta:=Up(p+1)$ give the  divisor $\Delta$. Note $\delta$ is equal to the $x$ of equation (\ref{p}). From the zeros of $\delta$ we get a fiber of type $I_5^*$
\[
2A+2\Theta_{1,2}+\Theta_{1,3}+\Theta_{1,1}+2\sum_3^6\Theta_{0,i}+\Theta_{0,2}+\Theta_{p,-1,3}.
\]
After an easy calculation we get a Weierstrass equation of the fibration

\begin{equation}
y^2={x}^{3}+\delta\,\left (1+4\,\delta\right ){x}^{2}+2\,{\delta}
^{4}x+{\delta}^{7}.
\label{delta}
\end{equation}

The singular fibers are
\[
\begin{array}
[c]{cccc}%
\text{at} & \delta=0 & \text{of type} & I_{5}^*\\
\text{at} & \delta=\infty   & \text{ of type } & II^*\\
\text{at} & \delta=-2   & \text{ of type } & I_{2}\\
\text{at} & \delta=-4/27   & \text{ of type } & I_{1}\\
\end{array}
\]

The Mordell-Weil group is equal to $0.$

\subsubsection{Fibration of parameter $\pi$}
On the previous figure (Figure \ref{fig: K1}) 
we can see the singular fiber 
\[
\Theta_{\infty,1}+\Theta_{\infty,7}+2\Theta_{\infty,0}+2(0)+2\Theta_{0,0}+2\Theta_{0,7}+2\Theta_{0,6}+(2A)+
2\Theta_{1,2}+\Theta_{1,1}+\Theta_{1,3}.
\]
Using the previous calculation for $\delta$ we see that it corresponds to a fibration of parameter $\pi=\frac{U(p+1)}{p}$. The zeros of $\pi$ correspond to a fiber of type $I_3^*$.
After an easy calculation we have a Weierstrass equation of the fibration
\begin{equation}
y^2=x^3+\pi(\pi^2-2\pi-2)x^2+\pi^2(2\pi+1)x.
\label{pi}
\end{equation}
The singular fibers are
\[
\begin{array}
[c]{cccc}%
\text{at} & \pi=0 & \text{of type} & I_{3}^*\\
\text{at} & \pi=\infty   & \text{ of type } & I_{6}^*\\
\text{at} & \pi=-\frac{1}{2}   & \text{ of type } & I_{2}\\
\text{at} & \pi=4   & \text{ of type } & I_{1}\\
\end{array}
\]

The Mordell-Weil group is isomorphic to $\mathbb Z/2\mathbb Z.$

\subsubsection{Fibration of parameter $\mu$}
From the fibration of parameter $p$ we can also join the $I_4^*$ and $I_2^*$ fibers.
Let $\mu=\frac{y}{p(x-p(p+1)^2)}$, with $y,x$ from equation (\ref{p}). After an easy calculation we obtain a Weierstrass equation of the fibration of parameter $\mu$
\begin{equation}
y^2+\mu^2(x-1)y=x(x-\mu^2)^2.
\label{mu}
\end{equation}
The singular fibers are
\[
\begin{array}
[c]{cccc}%
\text{at} & \mu=0 & \text{of type} & IV^*\\
\text{at} & \mu=\infty   & \text{ of type } & I_{10}\\
\text{at} & \mu=\pm 1   & \text{ of type } & I_{2}\\
\text{at} & \mu=\mu_0   & \text{ of type } & I_{1}\\
\end{array}
\]
where $\mu_0$ is a root of the polynomial $2x^2-27$.
The rank of the Mordell-Weil group is $1$, the torsion group is $0.$  The height of point $(\mu^2,0)$ is equal to $\frac{1}{15}$.  
\begin{remark} This fibration can also be obtained with the method of the first set and the parameter ${\frac {{X}^{2}\left (Y-1\right )\left (Z-1\right )\left (YZ+X\right )
}{\left (X-1\right )\left (X+Z\right )\left (X+Y\right )}}$ or with parameter $\frac{(Us-1)s}{s-1}.$
\end{remark}

\subsubsection{Fibration of parameter $\alpha$}
Let $\alpha=\frac{y}{p(x-p)}$.  After an easy calculation we have a Weierstrass equation of the fibration of parameter $\alpha$
 \begin{equation}
y^2+(\alpha^2+2)yx-\alpha^2y=x^2(x-1). 
 \end{equation}
 
 The singular fibers are
\[
\begin{array}
[c]{cccc}%
\text{at} & \alpha=0 & \text{of type} & I_0^*\\
\text{at} & \alpha=\infty   & \text{ of type } & I_{14}\\
\text{at} & \alpha=\alpha_0    & \text{ of type } & I_{1}\\
\end{array}
\]
where $\alpha_0$ is a root of the polynomial $2x^4+13x^2+64$.
The Mordell-Weil group is of rank one, the torsion group is $0.$ The height of $(0,0)$ is $\frac{1}{7}.$ 

\subsection{From fibration of parameter $\delta$}
We redraw the graph of the components of the singular fibers and sections of the fibration of parameter $\delta$, and lookfort subgraphs.

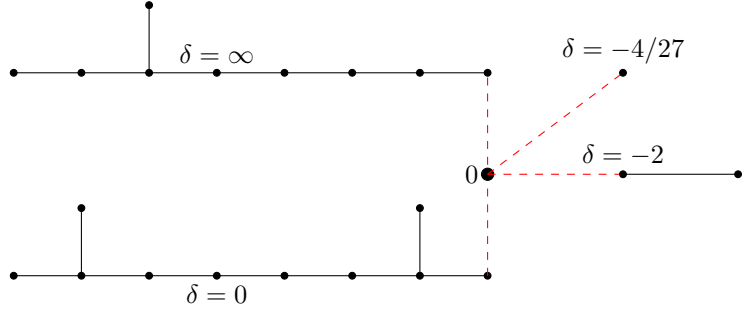
\begin{center}
\begin{tikzpicture}[scale=0.9]%
\draw[fill=black](0.1,1.1)circle (0.05cm);
\draw[fill=black](1.1,1.1)circle (0.05cm);
\draw[fill=black](2.1,1.1)circle (0.05cm);
\draw[fill=black](3.1,1.1)circle (0.05cm)node [below]{$\delta=0$};
\draw[fill=black](4.1,1.1)circle (0.05cm);
\draw[fill=black](5.1,1.1)circle (0.05cm);
\draw[fill=black](6.1,1.1)circle (0.05cm);
\draw[fill=black](7.1,1.1)circle (0.05cm);
\draw[fill=black](1.1,2.1)circle (0.05cm);
\draw[fill=black](6.1,2.1)circle (0.05cm);
\draw(0.1,1.1)--(7.1,1.1);
\draw(1.1,1.1)--(1.1,2.1);
\draw(6.1,1.1)--(6.1,2.1);

\draw(0.1,4.1)--(7.1,4.1);
\draw(2.1,4.1)--(2.1,5.1);
\draw[fill=black](0.1,4.1)circle (0.05cm);
\draw[fill=black](1.1,4.1)circle (0.05cm);
\draw[fill=black](2.1,4.1)circle (0.05cm);
\draw[fill=black](3.1,4.1)circle (0.05cm)node [above]{$\delta=\infty$};
\draw[fill=black](4.1,4.1)circle (0.05cm);
\draw[fill=black](5.1,4.1)circle (0.05cm);
\draw[fill=black](6.1,4.1)circle (0.05cm);
\draw[fill=black](7.1,4.1)circle (0.05cm);
\draw[fill=black](2.1,5.1)circle (0.05cm);

\draw[dashed,red](7.1,1.1)--(7.1,4.1);
\draw[fill=black](7.1,2.6)circle (0.09cm)node [left]{$0$};
\draw[dashed,red](7.1,2.6)--(9.1,2.6);
\draw[dashed,red](7.1,2.6)--(9.1,4.1);
\draw[fill=black](9.1,4.1)circle (0.05cm)node [above]{$\delta=-4/27$};;

\draw(9.1,2.6)--(10.8,2.6);
\draw[fill=black](9.1,2.6)circle (0.05cm)node [above]{$\delta=-2$};
\draw[fill=black](10.8,2.6)circle (0.05cm);

\end{tikzpicture}
\end{center} 

\subsubsection{Fibration of parameter $\beta$}
We can see the subgraph corresponding to a singular fiber of type $I_2^*$.

\begin{center}
\begin{tikzpicture}[scale=0.9]%
\draw[fill=black](5.1,1.1)circle (0.05cm);
\draw[fill=black](6.1,1.1)circle (0.05cm)node [below]{$\delta=0$};
\draw[fill=black](7.1,1.1)circle (0.05cm);
\draw[fill=black](6.1,2.1)circle (0.05cm);
\draw(5.1,1.1)--(7.1,1.1);
\draw(6.1,1.1)--(6.1,2.1);
\draw[fill=black](7.1,4.1)circle (0.05cm)node [above]{$\delta=\infty$};
\draw(7.1,1.1)--(7.1,4.1);
\draw[fill=red](7.1,2.6)circle (0.09cm)node [left]{$0$};
\draw(7.1,2.6)--(9.1,2.6);
\draw[fill=black](9.1,2.6)circle (0.05cm)node [above]{$\delta=-2$};
\end{tikzpicture}
\end{center}

To get a parameter $\beta$ corresponding to this fibration we do the transformation $x=u+\delta^3$ in equation (\ref{delta}) and obtain a new equation
\[
{y}^{2}-{u}^{3}-\delta\,\left (\delta+1\right )\left (3\,\delta+1
\right ){u}^{2}-{\delta}^{4}\left (\delta+2\right )\left (3\,\delta+2
\right )u-{\delta}^{7}\left (\delta+2\right )^{2}.
\]
The point $(0,0)$ is singular $\mod \delta$ and $\mod \delta+2$. By calculation we see that  $\beta=\frac{u}{\delta^2(\delta+2)}$ fits. We have  a Weierstrass equation
\begin{equation}
y^2=x^3+2\beta^2(\beta-1)x^2+\beta^3(\beta-1)^2 x.
\label{beta}
\end{equation}

The Mordell-Weil group is of order $2.$
 The singular fibers are
\[
\begin{array}
[c]{cccc}%
\text{at} & \beta=0 & \text{of type} & III^*\\
\text{at} & \beta=\infty   & \text{ of type } & I_{2}^*\\
\text{at} & \beta=1    & \text{ of type } & I_{1}^*\\
\end{array}
\]

\subsubsection{Fibration of parameter $\phi$}
We can see the subgraph corresponding to a singular fiber of type $I_7^*$.
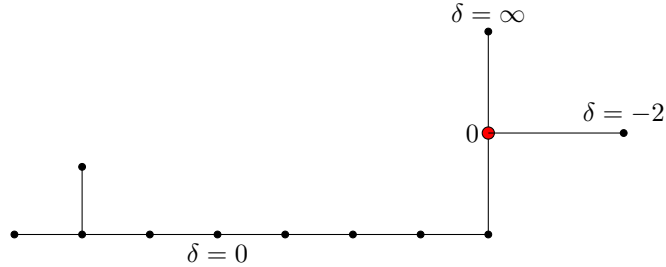
\begin{center}
\begin{tikzpicture}[scale=0.9]%
\draw[fill=black](0.1,1.1)circle (0.05cm);
\draw[fill=black](1.1,1.1)circle (0.05cm);
\draw[fill=black](2.1,1.1)circle (0.05cm);
\draw[fill=black](3.1,1.1)circle (0.05cm)node [below]{$\delta=0$};
\draw[fill=black](4.1,1.1)circle (0.05cm);
\draw[fill=black](5.1,1.1)circle (0.05cm);
\draw[fill=black](6.1,1.1)circle (0.05cm);
\draw[fill=black](7.1,1.1)circle (0.05cm);
\draw[fill=black](1.1,2.1)circle (0.05cm);
\draw(0.1,1.1)--(7.1,1.1);
\draw(1.1,1.1)--(1.1,2.1);

\draw[fill=black](7.1,4.1)circle (0.05cm)node [above]{$\delta=\infty$};

\draw(7.1,1.1)--(7.1,4.1);
\draw[fill=red](7.1,2.6)circle (0.09cm)node [left]{$0$};
\draw(7.1,2.6)--(9.1,2.6);

\draw[fill=black](9.1,2.6)circle (0.05cm)node [above]{$\delta=-2$};
\end{tikzpicture}
\end{center}
As previously, we start with the equation in $y,u$ and seek for a parameter of the form $\phi'=\frac{u}{\delta^2(\delta+2)}+\frac{a'}{\delta}$. We choose $a'$ to get an equation $y^2=P(u)$ with $P$ of degree $\leq 4$; we find $a=\frac{1}{2}$. Let $\phi=\phi'+1$, a Weierstrass equation is then 
\begin{equation}
y^2={x}^{3}+2\,{\phi}^{2}\left (4\,\phi-7\right ){x}^{2}-4\,{\phi}^{3}
\left (-3\,\phi+8\,{\phi}^{2}-4\right )x+8\,\left (3+4\,\phi\right ){
\phi}^{6}.
\end{equation}  

The singular fibers are
\[
\begin{array}
[c]{cccc}%
\text{at} & \phi=0 & \text{of type} & III^*\\
\text{at} & \phi=\infty   & \text{ of type } & I_{7}^*\\
\text{at} & \phi=\phi_0   & \text{ of type } & I_{1}\\
\end{array}
\]
where $\phi_0$ is a root of the polynomial $8x^2-13x+16$. 
The Mordell-Weil group is $0.$

The next table gives the corresponeance between parameters and elliptic fibrations.

\[%
\begin{array}
[c]{|c|c|c|c|c|}\hline
\text{parameter} & \text{singular fibers} & \text{type of reducible fibers%
}&\text{Rank}&\text{Torsion}\\\hline\hline
1-s & 2I_{8},I_{4},I_{2},2I_1 & A_{1},A_{3},A_{7},A_{7}&0&8\\\hline
2-k & I_{1}^{\ast},I_{12},I_{2},3I_{1} & A_{11},A_{1},D_{5}&1&4\\\hline
3-v & I_{8},I_{10},6I_{1} & A_{7},A_{9}&2&0\\\hline
4-a & I_{8},I_{1}^{\ast},I_6,3I_{1} & D_{5},A_{5},A_{7}&1&0\\\hline
5-d & 2I_{2}^{\ast},I_{2},I_{0}^{\ast} & A_{1},D_{4},2D_{6}&1&2\times 2\\\hline
6-p & I_{2}^{\ast},I_{4}^{\ast},I_{2},I_{4} & A_{1},D_{6},A_{3},D_{8}&0&2\times 2\\\hline
7-w & I_{6},I_{12},2I_{2},2I_{1} & A_{5},A_{1},A_{1}A_{11}&0&6\\\hline
8-b & 2IV^{\ast},I_{6},2I_{1} & A_{5},E_{6},E_{6}&1&3\\\hline
9-r & I_{6}^{\ast},I_{2}^{\ast},I_{2},2I_{1} & D_{6},A_{1},D_{10}&1&0\\\hline
10-e & III^{\ast},I_{4}^{\ast},2I_{2},I_{1} & A_{1},A_{1},D_{8},E_{7}&1&2\\\hline
11-f & III^{\ast},II^{\ast},I_{4},I_{1} & E_{7},A_{3},E_{8}&0&0\\\hline
12-g & 2III^{\ast},I_{4},I_{2} & E_{7},E_{7},A_{1},A_{3}&0&2\\\hline
13-h & 2II^{\ast},I_{2},2I_{1} & A_{1},E_{8},E_{8}&1&0\\\hline
14-t & 2I_{4}^{\ast},I_{2},2I_{1} & A_{1},D_{8},D_{8}&1&2\\\hline
15-l & I_{10},I_{3}^{\ast},5I_{1} & A_{9},D_{7}&2&0\\\hline
16-o & I_{12}^{\ast},I_{2},4I_{1} & A_{1},D_{16}&1&2\\\hline
17-q & I_{10}^{\ast},I_{4},I_{2},2I_{1} & A_{3},A_{1},D_{14}&0&2\\\hline
18-m & I_{16},3I_{2},2I_{1} & A_{1},A_{1},A_{1},A_{15}&0&4\\\hline
19-n & I_{16},I_{2},6I_{1} & A_{1},A_{15}&2&0\\\hline
20-j & IV^{\ast},I_{12},I_{2},2I_{1} & A_{11},E_{6},A_{1}&0&3\\\hline
21-c & I_{18},6I_1 & A_{17}&1&3\\\hline
22-u & I_{8}^{\ast},I_{1}^{\ast},I_{2},I_{1} & A_{1},D_{5},D_{12}&0&2\\\hline
23-i & I_{13}^{\ast},I_{2},3I_{1} & A_{1}D_{17}&0&0\\\hline
24-\psi &  III^{\ast},I_{6}^{\ast},3I_{1} & E_{7}D_{10}&1&2\\\hline
25-\delta &  I_{5}^{\ast},II^{\ast},I_{2},2I_{1} & E_{8},A_{1}D_{9}&0&0\\\hline
26-\pi &  I_{3}^{\ast},I_{6}^{\ast},I_{2},I_{1} & A_{1},D_{10},D_{7}&0&2\\\hline
27-\mu &  IV^{\ast},I_{10},2I_{2},2I_{1} & A_{9},A_{1},A_{1},E_{6}&1&0\\\hline
28-\alpha &  I_{0}^{\ast},I_{14},4I_{1} & D_{4},A_{13}&1&0\\\hline
29-\beta &  III^{\ast},I_{2}^{\ast},I_{1}^{\ast} & E_{7},D_{6},D_{5}&0&2\\\hline
30-\phi &  III^{\ast},I_{7}^{\ast},2I_{1} & E_{7},D_{11}&0&0\\\hline
\end{array}
\]

\end{document}